\documentclass[final,1p,times]{elsarticle}

\usepackage[colorlinks=true]{hyperref}
\usepackage{bbm}
\usepackage{amsfonts}
\usepackage{mathrsfs}
\usepackage{amsthm,amsmath,amssymb,mathrsfs,amsfonts,graphicx,graphics,latexsym,exscale,cmmib57,dsfont,bbm,amscd}
\usepackage[small,nohug,heads=vee]{diagrams}
\usepackage{color,soul}
\swapnumbers
\newtheorem*{Q}{Question}
\newtheorem*{1.7a}{(\ref{s1}.7a)}
\newtheorem*{1.7b}{(\ref{s1}.7b)}
\newtheorem*{1.7c}{(\ref{s1}.7c)}
\newtheorem*{1.7d}{(\ref{s1}.7d)}
\newtheorem*{1.7e}{(\ref{s1}.7e)}
\newtheorem*{1.7f}{(\ref{s1}.7f)}
\newtheorem*{1.7g}{(\ref{s1}.7g)}

\newtheorem*{(i)}{(i)}
\newtheorem*{(ii)}{(ii)}

\newtheorem*{t2.5a}{\ref{t2.5}a~Corollary}
\newtheorem*{t2.5b}{\ref{t2.5}b~Corollary}
\newtheorem*{t2.6a}{\ref{t2.6}a~Corollary}
\newtheorem*{t2.6b}{\ref{t2.6}b~Corollary}
\newtheorem*{2.5a}{\ref{2.5}a~Corollary}
\newtheorem*{2.5b}{\ref{2.5}b~Corollary}
\newtheorem*{2.5c}{\ref{2.5}c~Corollary}
\newtheorem*{2.7a}{\ref{2.7}a~Lemma}
\newtheorem*{2.7b}{\ref{2.7}b~Lemma}
\newtheorem*{t2.9a}{\ref{t2.9}a~Corollary}
\newtheorem*{t2.9b}{\ref{t2.9}b~Corollary}
\newtheorem*{t2.9c}{\ref{t2.9}c~Corollary}
\newtheorem*{t2.10a}{\ref{t2.10}a~Corollary}
\newtheorem*{t2.10b}{\ref{t2.10}b~Corollary}
\newtheorem*{t2.10c}{\ref{t2.10}c~Corollary}
\newtheorem*{t2.10d}{\ref{t2.10}d~Question}
\newtheorem*{t2.10e}{\ref{t2.10}e~Corollary}
\newtheorem*{t2.10f}{\ref{t2.10}f~Theorem}
\newtheorem*{t2.10g}{\ref{t2.10}g~Corollary}
\newtheorem*{4.1b}{\ref{4.1}b~Lemma}
\newtheorem*{4.2c}{\ref{4.2}c~Lemma}
\newtheorem*{4.2d}{\ref{4.2}d~Theorem}
\newtheorem*{4.2e}{\ref{4.2}e~Corollary}
\newtheorem*{4.2h}{\ref{4.2}h~Corollary}
\newtheorem*{4.2j}{\ref{4.2}j~Theorem}
\newtheorem*{4.3a}{\ref{4.3}a~Lemma}
\newtheorem*{4.3b}{\ref{4.3}b~Theorem}
\newtheorem*{4.3c}{\ref{4.3}c~Corollary}
\newtheorem*{4.3d}{\ref{4.3}d~Corollary}
\newtheorem*{cor*}{Corollary}
\newtheorem{thm}{Theorem}[section]
\newtheorem{lem}[thm]{Lemma}
\newtheorem{prop}[thm]{Proposition}

\newtheorem*{6.3a}{\ref{6.3}a~Lemma}
\newtheorem*{6.3b}{\ref{6.3}b~Theorem}
\newtheorem*{6.3c}{\ref{6.3}c~Proposition}
\newtheorem*{6.3d}{\ref{6.3}d~Proposition}

\theoremstyle{definition}

\newtheorem*{3.6}{\ref{3}.6~The Ellis-Glasner-Shapiro tower}

\newtheorem*{6.3e}{\ref{6.3}e~Remark}
\newtheorem*{Sn}{Standing notation}
\newtheorem*{1.7}{(\ref{s1}.7)~Remark}

\newtheorem{exa}[thm]{Example}
\newtheorem{se}[thm]{}
\newtheorem*{4.1a}{\ref{4.1}a~Definition}
\newtheorem*{4.2a}{\ref{4.2}a~Definition}
\newtheorem*{4.2b}{\ref{4.2}b~Definition}
\newtheorem*{4.2f}{\ref{4.2}f}
\newtheorem*{4.2g}{\ref{4.2}g}
\newtheorem*{4.2i}{\ref{4.2}i~Remark}
\newtheorem*{4.2k}{\ref{4.2}k~Remark}
\newtheorem*{4.3e}{\ref{4.3}e~Remark}
\newtheorem*{4.3f}{\ref{4.3}f~Remark}
\newtheorem*{4.3g}{\ref{4.3}g~Remark}
\newcommand{\I}{\mathbb{I}}

\newcommand{\bt}{\boldsymbol{t}}
\journal{JDE (23rd January 2023)}

\begin{document}

\begin{frontmatter}
\title{On M-dynamics and Li-Yorke chaos of extensions of minimal dynamics}

\author{Xiongping Dai}
\ead{xpdai@nju.edu.cn}
\address{Department of Mathematics, Nanjing University, Nanjing 210093, People's Republic of China}

\begin{abstract}
Let $\pi\colon\mathscr{X}\rightarrow\mathscr{Y}$ be an extension of minimal compact metric flows such that $\texttt{R}_\pi\not=\Delta_X$. A subflow of $\texttt{R}_\pi$ is called an M-flow if it is T.T. and contains a dense set of a.p. points. In this paper we mainly prove the following:
\begin{enumerate}[(1)]
\item $\pi$ is PI iff $\Delta_X$ is the unique M-flow containing $\Delta_X$ in $\texttt{R}_\pi$.
\item If $\pi$ is not PI, then there exists a canonical Li-Yorke chaotic M-flow in $\texttt{R}_\pi$. In particular, an Ellis weak-mixing non-proximal extension is non-PI and so Li-Yorke chaotic.
\end{enumerate}
In addition, we show
\begin{enumerate}[(1)]
\item[(3)] a unbounded or non-minimal M-flow, not necessarily compact, is sensitive; and
\item[(4)] every syndetically distal compact flow is pointwise Bohr a.p.
\end{enumerate}
\end{abstract}

\end{frontmatter}
\section{Introduction}\label{s1}
We begin by reviewing briefly the basic notions needed.
Let $T$ be a Hausdorff topological group with identity $e$ and $X$ a Hausdorff uniform space. Unless specified otherwise, by $(T,X)$ we mean a \textit{flow} \cite{GH, G76, V77, B79, A88, EE} with phase group $T$ and with phase space $X$; that is to say, there is a continuous phase mapping $T\times X\rightarrow X$, denoted ${(t,x)\mapsto tx}$, such that $ex=x$ and $(st)x=s(tx)$ for all $x\in X$ and $s,t\in T$. If $T$ is only a topological monoid, then $(T,X)$ will be called a \textit{semiflow}. If $X$ is compact (resp.~metrizable, $\dotsc$), then $(T,X)$ will be called a compact (resp.~metric, $\dotsc$) \textit{dynamic} where $T$ may be a group or monoid.

As in \cite{V77} the dynamics $(T,X)$, $(T,Y)$, $(T,Z)$, $\dotsc$ will be simply written as $\mathscr{X}$, $\mathscr{Y}$, $\mathscr{Z}$, $\dotsc$, respectively. Write $\mathscr{X}\times\mathscr{X}$ for $(T,X\times X)$ defined by $(t,(x_1,x_2))\mapsto(tx_1,tx_2)$.

\begin{Sn}
\item
\begin{enumerate}[1.]
\item In a non-metric space, ``$x_j\to x$'' is always under the sense of Moore-Smith net convergence.
\item By $\mathfrak{N}_x$ we will denote the neighborhood system of a point $x$ in the ambience.
\item $\mathscr{U}_X$ stands for a compatible symmetric uniformity structure of $X$. For all $A\subset X$ and $\varepsilon\in\mathscr{U}_X$, $\varepsilon[A]$ is the $\varepsilon$-neighborhood of $A$ in $X$ defined as the set $\{x\in X\,|\,\exists\,a\in A\textrm{ s.t. }(a,x)\in\varepsilon\}$.
\item $\{*\}$ stands for a one-point dynamic.
\end{enumerate}
\end{Sn}

Let $\mathscr{X}$ be any dynamic. Then $\mathscr{X}$  is minimal iff $\overline{Tx}=X\ \forall x\in X$. A point $x\in X$ is said to be \textit{almost periodic} (a.p.) under $\mathscr{X}$ provided that
\begin{enumerate}[({\ref{s1}.1}a)]
\item $N_T(x,U):=\{t\,|\,tx\in U\}$
\end{enumerate}
is syndetic in $T$ for all $U\in\mathfrak{N}_x$ (see Def.~I in $\S$\ref{s2.2.1}).
$\mathscr{X}$ is called \textit{topologically transitive} (T.T.) if for all nonempty open sets $U,V$ in $X$,
\begin{enumerate}[({\ref{s1}.1}b)]
\item $N_T(U,V):=\{t\in T\,|\,V\cap tU\not=\emptyset\}$
\end{enumerate}
is nonempty.
We say $\mathscr{X}$ is \textit{weakly mixing} if $\mathscr{X}\times\mathscr{X}$ is a T.T. dynamic.
Following Glasner-Weiss~\cite{GW}:
\begin{enumerate}[({\ref{s1}}.2)]
\item $\mathscr{X}$ is referred to as an \textit{M-dynamic} if $\mathscr{X}$ is T.T. and $X$ contains a dense set of a.p. points.
\end{enumerate}
See Theorems~\ref{thm1} and \ref{thm2} for a sufficient condition of M-dynamics. Moreover, we will prove that every syndetically distal compact flow is pointwise Bohr a.p. (see Def.~III and Def.~v in $\S$\ref{s2.2.1} and Theorem~\ref{t2.10}).

Let $\mathscr{X}$ and $\mathscr{Y}$ be two dynamics. A continuous map $\pi\colon X\rightarrow Y$ is called an \textit{extension} of dynamics, denoted $\pi\colon\mathscr{X}\rightarrow\mathscr{Y}$, if $\pi X=Y$ and $\pi tx=t\pi x$ for all $t\in T$ and $x\in X$. Extensions are important elements in the structure theory of minimal compact dynamics (cf.~\cite{G76, V77, B79, A88, EE, CD20}). For $\pi\colon\mathscr{X}\rightarrow\mathscr{Y}$ and $n\ge2$, we write
\begin{enumerate}[(\ref{s1}.3)]
\item $\texttt{R}_\pi^n=\{(x_1,\dotsc,x_n)\in X^n\,|\,\pi x_1=\dotsm=\pi x_n\}$ and $\texttt{R}_\pi=\texttt{R}_\pi^2$.
\end{enumerate}
Clearly, $\texttt{R}_\pi$ is an invariant closed equivalence relation on $X$. So $\mathscr{R}_\pi=(T,\texttt{R}_\pi)$ is a dynamic. Usually one is interested to the dynamics of $\mathscr{R}_\pi$ that is driven by $\mathscr{Y}$, for example, in skew-product flows and random dynamical systems.

Now let $\pi\colon\mathscr{X}\rightarrow\mathscr{Y}$ be an extension of minimal compact dynamics. As usual we say that:

\begin{enumerate}[({\ref{s1}}.4)]
\item[({\ref{s1}}.4)] $\pi$ is \textit{proximal} if $\overline{Tz}\cap\Delta_X\not=\emptyset$ $\forall z\in \texttt{R}_\pi$.
\item[({\ref{s1}}.5)]  $\pi$ is \textit{almost periodic} (a.p.) if there is no $(x,x^\prime)\in \texttt{R}_\pi\setminus\Delta_X$ such that $\overline{T(U\times V\cap \texttt{R}_\pi)}\cap\Delta_X\not=\emptyset$ for all $U\in\mathfrak{N}_x$ and $V\in\mathfrak{N}_{x^\prime}$.
\end{enumerate}
\begin{enumerate}[({\ref{s1}.6}a)]
\item $\pi$ is \textit{PI}~\cite{EGS, G76, V77, B79, A88, Ge01}, provided that there exists a minimal proximal extension $\rho\colon\mathscr{X}^\prime\rightarrow\mathscr{X}$ with $\pi^\prime=\pi\circ\rho\colon\mathscr{X}^\prime\rightarrow\mathscr{Y}$ can be built by
successive proximal and a.p. extensions; that is, there exists a ``PI-tower'':
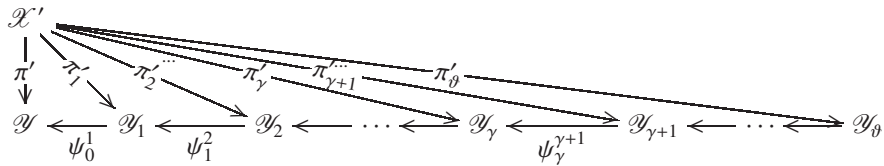
\begin{diagram}
\mathscr{X}^\prime\\
\dTo~{\pi^\prime}&\rdTo~{\pi_1^\prime}\rdTo(4,2)~{\pi_2^\prime}{{}_{\quad\ldots}}\rdTo(8,2)~{\pi_\gamma^\prime}\rdTo(10,2)~{\pi_{\gamma+1}^\prime}{{}_{\ldots}}\rdTo(14,2)~{\pi_\vartheta^\prime}\\
\mathscr{Y}&\lTo_{\ \psi_0^1}&\mathscr{Y}_1&\lTo_{\quad \psi_1^2\quad}&\mathscr{Y}_2&\lTo&\cdots&\lTo&\mathscr{Y}_\gamma&\lTo_{\quad \psi_\gamma^{\gamma+1}\quad}&\mathscr{Y}_{\gamma+1}&\lTo&\cdots&\lTo&\mathscr{Y}_\vartheta
\end{diagram}
such that:
\begin{enumerate}[a.]
\item $\pi_\vartheta^\prime$ is an isomorphism,
\item $\psi_\gamma^{\gamma+1}$ is proximal or a.p. for all ordinal $\gamma<\vartheta$, and
\item if $\gamma$ with $\gamma\le\vartheta$ is a limit ordinal then $\mathscr{Y}_\gamma=\underleftarrow{\lim}_{\lambda<\gamma}\mathscr{Y}_\lambda$.
\end{enumerate}
Here $\mathscr{X}^\prime$ and $\mathscr{Y}_\gamma$, for $\gamma\le\vartheta$, all are compact dynamics.

\item $\pi$ is \textit{strictly PI} in the special case that $\rho$ is an isomorphism in (\ref{s1}.6a); that is, $\pi$ possesses a PI-tower itself.

\item In particular, if $\mathscr{X}$ is a nontrivial minimal weakly mixing compact dynamics with $T$ nilpotent (\cite[Def.~2.3.8]{CD20}), then $\mathscr{X}\rightarrow \{*\}$ is not PI (see Lemma~\ref{3.12}).
\end{enumerate}

Every non-minimal T.T. dynamic with a dense set of periodic points is Devaney chaotic (cf., e.g.,~\cite{B92, HY, CC, SK}). In fact, a non-minimal compact M-dynamic is sensitive (cf.~\cite[Thm.~1.3]{GW} and \cite[Prop.~2.5]{DT}). See \cite{Z, YZ} on sensitivity and its variations of extensions of minimal compact metric flows.

In this paper, we shall study sensitivity of noncompact M-dynamics (see $\S$\ref{4}); moreover, we will define and study the Li-Yorke chaos of extensions of minimal compact metric dynamics from the viewpoint of M-dynamics (see $\S$\ref{5}); and prove that every non-PI extension of minimal compact dynamics can canonically induce M-dynamics (Theorem~\ref{3.5}). We shall improve Bronstein's intrinsic characterization of PI-extensions (Theorem~\ref{3.10}). Finally, we shall consider examples in $\S$\ref{6}.

We now conclude the Introduction with a remark on an important open problem of Robert Ellis (1970s), which is relative closely to the definition (1.2) of M-dynamics.
\begin{1.7}
Let $\mathscr{X}$ be a T.T. pointwise a.p. compact dynamic, where the phase space $X$ is non-metrizable.

\begin{Q}[{R.~Ellis; cf.~\cite[p.~263]{EE}}]
Is $\mathscr{X}$ a minimal dynamic?
\end{Q}
\noindent
In some special cases, using Ellis's metric approaches the answer is affirmative ((\ref{s1}.7a), (\ref{s1}.7b), and (\ref{s1}.7c)).
\begin{1.7a}
If $T$ is $\sigma$-compact, then $\mathscr{X}$ is minimal.
\end{1.7a}

\begin{proof}
See \cite[Thm.~B.2.1]{CD20} for $T$ a semigroup; \cite[Prop.~4.23]{EE} for $T$ a countable group.
\end{proof}

\begin{1.7b}[{cf.~\cite[Prop.~1.9]{E78} and \cite[Prop.~4.24]{EE} for $T$ a group; \cite[Cor.~1.31]{AD} for $T$ a semigroup}]
If $\mathscr{X}$ is distal, then $\mathscr{X}$ is minimal.
\end{1.7b}

\begin{1.7c}[{cf.~\cite[Thm.~B.2.4]{CD20}}]
If $\theta\colon\mathscr{X}\rightarrow\mathscr{Z}$ is a distal extension with $\mathscr{Z}$ a minimal dynamic (see Def.~\ref{s2}.0f), then $\mathscr{X}$ is minimal.
\end{1.7c}

\begin{1.7d}
If $\mathscr{X}$ is weakly a.p. (in the sense of Gottschalk; i.e., $x\mapsto\overline{Tx}$ is continuous), then $\mathscr{X}$ is minimal.
\end{1.7d}

\begin{proof}
Obvious.
\end{proof}

\begin{1.7e}
If $\mathscr{X}$ is locally a.p. (Def.~II in $\S$\ref{s2.2.1}) such that $tX=X$ for all $t\in T$, then $\mathscr{X}$ is minimal.
\end{1.7e}

\begin{proof}
First we claim that $\texttt{P}(\mathscr{X})=\texttt{RP}(\mathscr{X})$ (see Def.~(\ref{s2}.0a) and (\ref{s2}.0b)). In fact, if $\mathscr{X}$ is a flow, this is \cite[Lem.~13(4)]{EG}. Now assume $\mathscr{X}$ is a semiflow. Let $Z\subset X$ and $x\in X$ such that $x$ is distal from $Z$ (i.e., there is an index $\alpha\in\mathscr{U}_X$ with $(tx,tz)\notin\alpha\ \forall t\in T, z\in Z$). Take $\beta\in\mathscr{U}_X$ with $\beta^3\subseteq\alpha$. Pick $U\in\mathfrak{N}_x$ and a syndetic set $A$ in $T$ such that $AU\subseteq\beta[x]$. Clearly, we have for all $a\in A$, $y\in U$ and $z\in Z$ that $(ay,az)\notin\beta$. Next, select a compact set $K$ in $T$ such that $Kt\cap A\not=\emptyset\ \forall t\in T$. Take $\gamma\in\mathscr{U}_X$ so small that $K\gamma\subseteq\beta$. Now for all $y\in U$, $z\in Z$ and $t\in T$, we have that $(ty,tz)\notin\gamma$; for otherwise, we can pick some $k\in K$ with $a=kt\in A$ so that $(ay,az)\in\beta$.
This implies that if $(x,x^\prime)\notin\texttt{P}(\mathscr{X})$, then there exist $U\in\mathfrak{N}_x$ and $V\in\mathfrak{N}_{x^\prime}$ such that $U$ is distal from $V$ under $\mathscr{X}$, and further, $(x,x^\prime)\notin\texttt{RP}(\mathscr{X})$. Thus, $\texttt{P}(\mathscr{X})=\texttt{RP}(\mathscr{X})$. (\textbf{Note.} This claim does not need conditions $tX=X\ \forall t\in T$ and T.T.)

If $M_1\not=M_2$ are two minimal sets in $X$, then $tM_1=M_1$ for all $t\in T$, and further by T.T. it follows that there is a pair $(x_1,x_2)\in M_1\times M_2$ such that $(x_1,x_2)\in\texttt{RP}(\mathscr{X})$. Then $(x_1,x_2)\in\texttt{P}(\mathscr{X})$ and $\overline{Tx_1}=\overline{Tx_2}$, contrary to $M_1\not=M_2$. Thus, $\mathscr{X}$ is minimal.
\end{proof}

\begin{1.7f}
If $\mathscr{X}$ is a pointwise regularly a.p. \textbf{flow} (see Def.~V in $\S$\ref{s2.2.1}), then $\mathscr{X}$ is minimal regularly a.p.
\end{1.7f}

\begin{proof}
By \cite[Thm.~5.18]{GH}, $(T,\overline{Tx})$ is minimal regularly a.p. for all $x\in X$ so that $\mathscr{X}$ is distal. Then by (\ref{s1}.7b), $\mathscr{X}$ is a minimal regularly a.p. flow.
\end{proof}

\begin{1.7g}
$\mathscr{X}$ is either minimal or sensitive (see Def.~\ref{4.2}b).
\end{1.7g}

\begin{proof}
By Theorem~\ref{4.2}j in $\S$\ref{4}.
\end{proof}
\end{1.7}

\section{AP-Transitive relations and stability}\label{s2}
We shall present a sufficient necessary condition for a subdynamic of $\mathscr{R}_\pi$ to be an M-dynamic in $\S$\ref{s2.1}. Moreover, we shall study in $\S$\ref{s2.2} the stability of noncompact flows.
\subsection{AP-Transitive relations}\label{s2.1}
In this subsection, $T$ is thought of as a discrete group or semigroup; and let $\mathbb{I}$ be any fixed minimal left ideal in $\beta T$ and $\mathbb{J}=\{u\in\mathbb{I}\,|\,u^2=u\}$.

Let $\pi\colon\mathscr{X}\rightarrow\mathscr{Y}$ be an extension of compact dynamics, where $\mathscr{Y}$ is minimal.
We first need to introduce a relation on $X$. Let $L\subseteq \texttt{R}_\pi$, $L\not=\emptyset$.
\begin{enumerate}[$\bullet$]
\item We say that $L$ is an \textit{AP-transitive relation for $\pi$} on $X$, provided that $L$ is a reflexive symmetric relation on $X$ such that if $(x_1,x_2)$, $(x_2,x_3)\in L$ are both a.p. points, then $(x_1,x_3)\in L$.
\end{enumerate}

Clearly, an invariant closed AP-transitive relation for $\pi$ on $X$ need not be an equivalence relation. Of course, if $\mathscr{X}$ is a $\pi$-distal extension of $\mathscr{Y}$, then every AP-transitive relation for $\pi$ is an equivalence relation on $X$.
\begin{enumerate}[$\bullet$]
\item Let $L=\overline{\{z\in \texttt{R}_\pi\,|\,z \textrm{ is a.p.}\}}$ such that if $u,v\in\mathbb{J}$ with $ux=vx$ for some $x\in X$ then $ux^\prime=vx^\prime$ for all $x^\prime\in L[x]\, (=\{x^\prime\in X\,|\,(x,x^\prime)\in L\})$. Then $L$ is an AP-transitive relation for $\pi$ on $X$.
\end{enumerate}
\begin{proof}
Let $(x_1,x_2)$ and $(x_2,x_3)\in L$ be a.p. Then there are $u,v\in\mathbb{J}$ with $u(x_1,x_2)=(x_1,x_2)$ and $v(x_2,x_3)=(x_2,x_3)$. So $ux_2=vx_2$. Then $ux_1=vx_1$ and $(x_1,x_3)=v(x_1,x_3)\in \texttt{R}_\pi$ is a.p. so that $(x_1,x_3)\in L$.
\end{proof}

Let $\texttt{P}_\pi$ and $\texttt{RP}_\pi$ be the $\pi$-relative proximal and regionally proximal relations on $X$, respectively; that is,
$$\texttt{P}_\pi=\{\tilde{x}\in \texttt{R}_\pi\,|\,\exists\, t_j\in T, \tilde{x}^\prime\in\Delta_X\textrm{ s.t. }t_j\tilde{x}\to\tilde{x}^\prime\}.\leqno{(\textrm{\ref{s2}.0a})}$$
and
$$\texttt{RP}_{\pi}=\{\tilde{x}\in \texttt{R}_\pi\,|\,\exists\, \tilde{x}_j\in \texttt{R}_\pi, t_j\in T, \tilde{x}^\prime\in\Delta_X\ s.t.\ \tilde{x}_j\to\tilde{x}, t_j\tilde{x_j}\to\tilde{x}^\prime\}.\leqno{(\textrm{\ref{s2}.0b)}}$$
In the case of $\mathscr{Y}=\{*\}$, we will write $\texttt{P}(\mathscr{X})=\texttt{P}_\pi$ and $\texttt{RP}(\mathscr{X})=\texttt{RP}_\pi$.
Let $L\subseteq \texttt{R}_\pi$, we then define
$$\texttt{P}_{\pi|L}=\texttt{P}_\pi\cap L\leqno{(\textrm{\ref{s2}.0c)}}$$
and
$$
\texttt{RP}_{\pi|L}=\{\tilde{x}\in L\,|\,\exists\, \tilde{x}_j\in L, t_j\in T, \tilde{x}^\prime\in\Delta_X\ s.t.\ \tilde{x}_j\to\bar{x}, t_j\tilde{x_j}\to\tilde{x}^\prime\}.\leqno{(\textrm{\ref{s2}.0d)}}
$$
In the case of $L=\texttt{R}_\pi$, $\texttt{P}_{\pi|L}=\texttt{P}_\pi$ and $\texttt{RP}_{\pi|L}=\texttt{RP}_\pi$.

Then:
\begin{enumerate}[{\ref{s2}.0}e.]
\item $\pi$ is \textit{proximal} iff $\texttt{P}_\pi=\texttt{R}_\pi$;
\item[{\ref{s2}.0}f.] a point $x\in X$ is \textit{$\pi$-distal} iff $\texttt{P}_\pi[x]\cap\overline{Tx}=\{x\}$; $\pi$ is \textit{distal} iff $\texttt{P}_\pi=\Delta_X$; in the case $\mathscr{Y}=\{*\}$, $\mathscr{X}$ is \textit{distal} iff $\pi$ is distal iff $\texttt{P}(\mathscr{X})=\Delta_X$.
\item[{\ref{s2}.0}g.] $\pi$ is \textit{distal-equicontinuous} (d.e.; cf.~\cite{CD20}) iff $\texttt{RP}_{\pi}=\Delta_X$ iff $\pi$ is a.p. (noting that $\pi$-equicontinuity need not imply $\pi$-distality in compact semiflows~\cite[Thm.~1.5, Exa.~1.6]{DLX}).

\item[{\ref{s2}.0}h.] If $\mathscr{X}$ is a compact flow with $\texttt{P}(\mathscr{X})=\texttt{RP}(\mathscr{X})$, then 
$$
\texttt{P}(\mathscr{X})=\{(x,x^\prime)\,|\,\forall \delta\in\mathscr{U}_X\exists A\subseteq T\textrm{ a syndetic set s.t. }A(x,x^\prime)\subseteq\delta\}
$$
(see \cite[Thm.~7]{C1} or \cite[Thm.~1.3.14]{B79}).
\end{enumerate}
Here the $\pi$-proximal cell of $x\in X$ is defined by $\texttt{P}_\pi[x]=\{x^\prime\,|\,x^\prime\in X, (x,x^\prime)\in \texttt{P}_\pi\}$. Notice that d.e. extension is also called a.p. extension or isometric extension in flows~(cf., e.g., \cite{G76, V77, A88}).

\begin{thm}[{cf.~\cite{M76, V77} for $L=\texttt{R}_\pi$}]\label{thm1}
Let $\pi\colon\mathscr{X}\rightarrow\mathscr{Y}$ be an extension of minimal compact flows. Suppose $L\subseteq \texttt{R}_\pi$ is a closed invariant AP-transitive relation on $X$ such that
\begin{enumerate}[(a)]
\item $L$ has a dense set of a.p. points and
\item $\texttt{RP}_{\pi|L}=L$.
\end{enumerate}
Then $(T,L)$ is T.T.; that is, $(T,L)$ is an M-flow.
\end{thm}

Since $L$ is not necessarily an equivalence relation on $X$, so $T\times X/L\rightarrow X/L$, defined by $(t,L[x])\mapsto L[tx]$, need not be a well-defined phase mapping. In view of this, Theorem~\ref{thm1} is not a corollary of the classical McMahon-Veech theorem.
In fact, using Ellis's algebraic theory we can obtain the semiflow version of the above theorem as follows:

\begin{thm}[{cf.~\cite{CD20} for $L=\texttt{R}_\pi$}]\label{thm2}
Let $\pi\colon\mathscr{X}\rightarrow\mathscr{Y}$ be an extension of minimal compact semiflows. Suppose $L\subseteq \texttt{R}_\pi$ is a closed invariant AP-transitive relation on $X$ such that
\begin{enumerate}[(a)]
\item $L$ has a dense set of a.p. points and
\item $\texttt{RP}_{\pi|L}=L$.
\end{enumerate}
Then $(T,L)$ is an M-semiflow.
\end{thm}

Before proving Theorem~\ref{thm2} we need to introduce some terms for our convenience.
Let $\textrm{Aut}\,(\mathscr{I})$ be the set of automorphisms of the universal minimal compact dynamic $\mathscr{I}=(T,\mathbb{I})$. Put
$$\Gamma[\alpha]=\{(m,\alpha m)\,|\,m\in\I\}\quad \forall \alpha\in\textrm{Aut}\,(\mathscr{I}).$$
The so-called \textit{$\tau$-topology} on $\textrm{Aut}\,(\mathscr{I})$ is defined as follows:

Let $\alpha_i\in\textrm{Aut}\,(\mathscr{I})$ be a net and let $\alpha\in\textrm{Aut}\,(\mathscr{I})$. We say $\alpha_i\to_\tau\alpha$ iff for every $m\in\I$ there exists a net $m_i\to m$ in $\I$ such that $\alpha_im_i\to\alpha m$ in $\I$ (cf.~\cite[Appendix~A]{CD20}).

Under the $\tau$-topology $\textrm{Aut}\,(\mathscr{I})$ is a compact $T_1$-space (cf.~\cite{G76, A88} and \cite[Prop.~A.2.5]{CD20}). Let $\mathbb{F}$ be a $\tau$-closed subgroup of $\textrm{Aut}\,(\mathscr{I})$; then its \textit{derived subgroup} is defined as follows:
$$
\mathbb{F}^\prime=\{\alpha\in\mathbb{F}\,|\,\exists\,\delta_i\in\mathbb{F}\textrm{ s.t. }\delta_i\to_\tau\alpha\ \&\ \delta_i\to_\tau \textrm{id}_\I\}
$$
(cf.~\cite[Def.~A.4.1]{CD20}). Here $\mathbb{F}^\prime$ measures clearly the degree to which the $\tau$-topology on $\mathbb{F}$ fails to be a Hausdorff space. Moreover, $\mathbb{F}^\prime$ is a $\tau$-closed normal subgroup of $\mathbb{F}$ (\cite[Lem.~A.4.2]{CD20}).

\begin{proof}[\textbf{Proof of Theorem~\ref{thm2}}]
At first we can construct a CD of minimal compact semiflows and extensions as follows:
\begin{diagram}
\mathscr{I}&\rTo^{\quad\pi_X\quad}&\mathscr{X}\\
&\rdTo_{\pi_Y}&\dTo_\pi\\
&&\mathscr{Y}
\end{diagram}
and set
$$
\mathbb{F}=\{\alpha\in \textrm{Aut}\,(\mathscr{I})\,|\,\pi_Y=\pi_Y\circ\alpha\};
$$
that is, the Ellis group of $\mathscr{Y}$ rel. $\pi_Y$. Let
$$
\mathbb{L}=\{\alpha\in\mathbb{F}\,|\,(\pi_X\times\pi_X)\Gamma[\alpha]\subseteq L\}.
$$
Notice that $\Gamma[\alpha]$ consists of a.p. points of $\texttt{R}_\pi$ for every $\alpha\in\mathbb{F}$. Since $L$ is an AP-transitive relation on $X$, $(\alpha^{-1}m,m)=(\alpha^{-1}m,\alpha(\alpha^{-1}m))\ \forall m\in\mathbb{I}$, and $\mathbb{F}$ is a group, hence $\mathbb{L}$ is a subgroup of $\mathbb{F}$. Moreover, $\mathbb{L}$ is $\tau$-closed. Indeed, if $\alpha_j\in\mathbb{L}$ and $\alpha_j\to_\tau\alpha$, then there exists a net $m_j\to m$ in $\mathbb{I}$ such that $\alpha_jm_j\to\alpha m$ so $(\pi_X\times\pi_X)(m,\alpha m)\in L$ and $\alpha\in\mathbb{L}$, for $L$ is closed. Thus, $\mathbb{L}$ is a $\tau$-closed subgroup of $\mathbb{F}$.

Let $(x,x^\prime)\in L$ be an a.p. point. There exists $u\in\mathbb{J}$ with $u(x,x^\prime)=(x,x^\prime)$. Since $(x,x^\prime)\in\texttt{RP}_{\pi|L}$ and $L$ has a dense set of a.p. points, by a standard argument (cf., e.g.,~\cite[Proof of Thm.~3.1.3]{CD20}) we can select $(m,m^\prime)$, $(m_j,m_j^\prime)$, $(n,n^\prime)$ in $\texttt{R}_{\pi_Y}\subseteq\mathbb{I}\times\mathbb{I}$ and $t_j\in T$ with
$$
(\pi_X\times\pi_X)(m,m^\prime)=(x,x^\prime),\ (\pi_X\times\pi_X)(m_j,m_j^\prime)\in L,\ \pi_X n=\pi_Xn^\prime,
$$
and
$$
(m_j,m_j^\prime)\to(m,m^\prime), \ t_j(m_j,m_j^\prime)\to(n,n^\prime)
$$
such that $(m,m^\prime)$ and $(m_j,m_j^\prime)$ are all a.p. under $(T,\mathbb{I}\times\mathbb{I})$. Then by regularity of $\mathscr{I}$, there exist $\alpha_j, \gamma\in\mathbb{L}$ such that
$$
m^\prime=\gamma m,\ m_j^\prime=\alpha_jm_j,\ (m_j,\alpha_jm_j)\to(m,\gamma m), \ (t_jm_j,\alpha_jt_jm_j)\to(n,n^\prime).
$$
Now since $\pi_Xn=\pi_Xn^\prime$, we can take $n^{\prime\prime}\in\mathbb{I}$ such that $(n^\prime,n^{\prime\prime})$ is a.p., $n$ is proximal with $n^{\prime\prime}$ and $\pi_Xn^{\prime\prime}=\pi_Xn^\prime$. Thus, there exists some $\xi\in\textrm{Aut}\,(\mathscr{I})$ with $n^{\prime\prime}=\xi n^\prime$. Then $\pi_X\xi n^\prime=\pi_Xn^\prime$ so $\pi_X\circ\xi=\pi_X$, $\xi\in\mathbb{L}$, and $(m_j,\xi\alpha_jm_j)\to(m,\xi\gamma m)$. Furthermore, there is a net $s_j\in T$ such that
$(s_jm_j,\xi\alpha_js_jm_j)\to(n,n)$. Thus, $\xi\alpha_j\to_\tau\xi\gamma$, $\xi\alpha_j\to_\tau\textrm{id}_\mathbb{I}$, $\xi\alpha_j\in\mathbb{L}$. So $\xi\gamma\in\mathbb{L}^\prime$.

To sum up, we have concluded that if $(x,x^\prime)\in L$ is a.p., then there exists an $\alpha\in\mathbb{L}$ with $(x,x^\prime)\in(\pi_X\times\pi_X)\Gamma[\alpha]$ such that there is a net $\delta_j\in\mathbb{L}$ with $\delta_j\to_\tau\alpha$ and $\delta_j\to_\tau\textrm{id}_\mathbb{I}$; that is, $\alpha\in\mathbb{L}^\prime$.

Since $\mathbb{L}^\prime$ is a $\tau$-closed subgroup (cf.~\cite[Lem.~A.4.2]{CD20}), hence if $\alpha,\gamma\in\mathbb{L}^\prime$, then there exists a net $\delta_j\in\mathbb{L}$ such that
$\delta_j\to_\tau\alpha$ and $\delta_j\to_\tau\gamma$ (cf.~\cite[Lem.~A.4.3]{CD20}).

Now let $\bar{x}$, $\bar{w}$ be two a.p. points in $L$, and, let $U$ and $V$ be two neighborhoods of $\bar{x}$ and $\bar{w}$ in $L$, respectively. Then we can take $\alpha,\gamma\in\mathbb{L}^\prime$ and $m,n\in\mathbb{I}$ such that $(\pi_X\times\pi_X)(m,\alpha m)=\bar{x}$ and $(\pi_X\times\pi_X)(n,\gamma n)=\bar{w}$. So there exists a net $\delta_j\in\mathbb{L}$ such that $\delta_j\to_\tau\alpha$ and $\delta\to_\tau\gamma$. Further, there are nets $m_j\to m$ and $n_j\to n$ in $\mathbb{I}$ such that $\delta_jm_j\to\alpha m$ and $\delta n_j\to\gamma n$. Since $(\pi_X\times\pi_X)(m_j,\delta_jm_j)$ and $(\pi_X\times\pi_X)(n_j,\delta_jn_j)$ lie in a same minimal subset of $L$, this implies that $N_T(U,V)\not=\emptyset$. As $L$ has already a dense set of a.p. points, $(T,L)$ is T.T. and an M-semiflow. The proof is complete.
\end{proof}

As a matter of fact, if $(T,L)$ is an M-semiflow, then conditions (a) and (b) in Theorem~\ref{thm2} are clearly fulfilled.

\subsection{Lyapunov stability and Bohr/regular almost periodicity}\label{s2.2}
In this part, let $\mathscr{X}$ be an arbitrary flow, not necessarily minimal and not necessarily compact, unless specified otherwise. We will now consider conditions under which an (regularly) a.p. point is Lyapunov stable.
\subsubsection{Basic definitions}\label{s2.2.1}
We first introduce some notions needed later.  Let $S$ and $A$ be subsets of $T$. Then $S$ is said to be (right) \textit{thick} if for all compact subset $K$ of $T$ there exists an element $t\in T$ such that $Kt\subseteq S$; $A$ is called a (right) \textit{syndetic subset} of $T$ if there exists a compact subset $K$ of $T$ such that $T=K^{-1}A$, or equivalently, $Kt\cap A\not=\emptyset\ \forall t\in T$.
It is easy to check
\begin{enumerate}[$\bullet$]
\item  \textit{A subset of $T$ is syndetic iff it intersects non-voidly every thick subset of $T$.}
\item \textit{If $T$ is an abelian group and $B$ syndetic and $S$ thick in $T$, then $T=BS$.}
\end{enumerate}

\noindent
Recall that:
\begin{enumerate}[I.]
\item A point $x\in X$ is \textit{almost periodic} (a.p.) under $\mathscr{X}$ iff $N_T(x,U)$ is a syndetic subset of $T$ for every $U\in\mathfrak{N}_x$ in $X$. Then $x$ is a.p. under $\mathscr{X}$ iff $x\in\overline{Sx}$ for all $S\subset T$ thick.
\item A point $x\in X$ is \textit{locally a.p.} under $\mathscr{X}$ iff for every $U\in\mathfrak{N}_x$ there exists a $V\in\mathfrak{N}_x$ and a syndetic subset $A$ of $T$ such that $AV\subseteq U$. We say $\mathscr{X}$ is \textit{locally a.p.} if it is pointwise locally a.p. under $\mathscr{X}$.
\item $\mathscr{X}$ is called \textit{a.p.} if for every $\alpha\in\mathscr{U}_X$ there exists a syndetic set $A$ in $T$ such that $Ax\subseteq\alpha[x]$ for all
$x\in X$. A point $x\in X$ is \textit{Bohr a.p.} under $\mathscr{X}$ if $(T,\overline{Tx})$ is an a.p. flow. We say $\mathscr{X}$ is \textit{pointwise Bohr a.p.} iff each point of $X$ is a Bohr a.p. point under $\mathscr{X}$.
\item $\mathscr{X}$ is called \textit{equicontinuous} if given $\epsilon\in\mathscr{U}_X$ and $x\in X$ there exists an index $\delta\in\mathscr{U}_X$ such that $t(\delta[x])\subseteq\epsilon[tx]$ for all $t\in T$. $\mathscr{X}$ is said to be \textit{uniformly equicontinuous} if for all $\varepsilon\in\mathscr{U}_X$ there exists an index $\delta\in\mathscr{U}_X$ with $T\delta\subseteq\varepsilon$.
\item $\mathscr{X}$ is \textit{regularly a.p.} at $x\in X$, or $x$ is a \textit{regularly a.p. point} under $\mathscr{X}$, iff $N_T(x,U)$ contains a syndetic normal closed subgroup of $T$ for all $U\in\mathfrak{N}_x$. We say $\mathscr{X}$ is \textit{regularly a.p.} iff given $\varepsilon\in\mathscr{U}_X$ there exists a syndetic normal closed subgroup $A$ of $T$ such that  $Ax\subseteq\varepsilon[x]$ for all $x\in X$. $\mathscr{X}$ is called \textit{point-regularly a.p.} iff there exists a regularly a.p. point $x$ such that $\overline{Tx}=X$.
\end{enumerate}

If $X$ is compact, then $x\in X$ is a.p. under $\mathscr{X}$ iff $\overline{Tx}$ is minimal under $\mathscr{X}$; and, $\mathscr{X}$ is a.p. iff it is uniformly equicontinuous iff it is equicontinuous iff $\texttt{RP}(\mathscr{X})=\Delta_X$ (see Lemma~\ref{l2.3}).

We say that
\begin{enumerate}[i.]
\item $\mathscr{X}$ is \textit{thickly stable} iff for every index $\varepsilon\in\mathscr{U}_X$ and all point $x\in X$ there exists an index $\delta\in\mathscr{U}_X$ and a thick subset $S$ of $T$ such that $s(\delta[x])\subseteq\varepsilon[sx]$ for all $s\in S$. If here $\delta$ and $S$ are independent of the choice of $x$, then $\mathscr{X}$ is called \textit{uniformly thickly stable}.
\item $\mathscr{X}$ is \textit{thickly regularly stable} iff for every index $\varepsilon\in\mathscr{U}_X$ and all point $x\in X$ there exists an index $\delta\in\mathscr{U}_X$ and a thick subsemigroup $S$ of $T$ such that $s(\delta[x])\subseteq\varepsilon[sx]$ for all $s\in S$.
\item $\mathscr{X}$ is \textit{uniformly thickly regularly stable} iff for every index $\varepsilon\in\mathscr{U}_X$ there exists an index $\delta\in\mathscr{U}_X$ and a thick subsemigroup $S$ of $T$ such that $S\delta\subseteq\varepsilon$.
\item $\mathscr{X}$ is \textit{syndetically stable} provided that for every index $\varepsilon\in\mathscr{U}_X$ and all point $x\in X$ there exists an index $\delta\in\mathscr{U}_X$ and a syndetic subset $S$ of $T$ such that $s(\delta[x])\subseteq\varepsilon[sx]$ for all $s\in S$. $\mathscr{X}$ is \textit{uniformly syndetically stable} iff for every index $\varepsilon\in\mathscr{U}_X$ there exists an index $\delta\in\mathscr{U}_X$ and a syndetic subset $S$ of $T$ such that $S\delta\subseteq\varepsilon$.
\item $\mathscr{X}$ is \textit{syndetically distal} if for every $x\in X$ and $\varepsilon\in\mathscr{U}_X$ there exists a syndetic subset $S$ of $T$ and an index $\delta\in\mathscr{U}_X$ such that if $y\notin\varepsilon[x]$ then $ty\notin\delta[tx]$ for all $t\in S$. $\mathscr{X}$ is \textit{uniformly syndetically distal} if for every $\varepsilon\in\mathscr{U}_X$ there exists a syndetic subset $S$ of $T$ and an index $\delta\in\mathscr{U}_X$ such that if $y\notin\varepsilon[x]$ then $ty\notin\delta[tx]$ for all $t\in S$ and $x\in X$.
\end{enumerate}
Let $S\subseteq T$ and $X_0\subseteq X$ be nonempty sets. Following \cite[Def.~2]{E} and \cite[Def.~2.8.6 and Def.~2.8.11]{B79} we say that:
\begin{enumerate}[a.]
\item $\mathscr{X}$ is \textit{Lyapunov $S$-stable} iff for every index $\varepsilon\in\mathscr{U}_X$ and all $x\in X$ there exists an index $\delta\in\mathscr{U}_X$ such that $t(\delta[x])\subseteq\varepsilon[tx]$ for all $t\in S$.
\item $\mathscr{X}$ is \textit{uniformly Lyapunov $S$-stable} iff for every index $\varepsilon\in\mathscr{U}_X$ there exists an index $\delta\in\mathscr{U}_X$ such that $S\delta\subseteq\varepsilon$.
\item $\mathscr{X}$ is \textit{Lyapunov $S$-stable at $L\subseteq X$ w.r.t. $X_0$} iff for every index $\varepsilon\in\mathscr{U}_X$ and all $x\in L$ there exists an index $\delta\in\mathscr{U}_X$ such that $t(\delta[x]\cap X_0)\subseteq\varepsilon[tx]$ for all $t\in S$.
\end{enumerate}

Clearly the uniformly Lyapunov $T$-stable is exactly the uniformly equicontinuous as in Def.~IV. If $X$ is compact, then Def.~a and Def.~b are equivalent.

Notice that comparing with ``Lyapunov $S$-stable'' and ``uniformly Lyapunov $S$-stable'' of England~\cite[Def.~2]{E} and Bronstein~\cite[2.8.6 and 2.8.11]{B79}, $S$ may vary with $\varepsilon$ in our thick and syndetic stability cases i., ii., iii., iv., and v. above. In addition, if $T$ is not abelian, the inverse of a syndetic (resp.~thick) set need not be syndetic (resp.~thick).

The following important result is due to W.H. Gottschalk. However, we will give an alternative simple proof here for reader's convenience.

\begin{lem}[{cf.~\cite{Go}; also \cite[Thm.~1.6.23]{B79} and \cite[Thm.~2.2]{A88}}]\label{l2.3}
A compact flow $\mathscr{X}$ is equicontinuous (or equivalently, Lyapunov $T$-stable) iff it is an a.p. flow.
\end{lem}

\begin{proof}
Assume $\mathscr{X}$ is an a.p. flow. Clearly, it is uniformly syndetically stable and distal so that $\texttt{RP}(\mathscr{X})=\texttt{P}(\mathscr{X})=\Delta_X$. Thus $\mathscr{X}$ is equicontinuous for $\texttt{RP}(\mathscr{X})=\bigcap_{\alpha\in\mathscr{U}_X}\overline{T\alpha}$ and $X$ compact.
Conversely, suppose $\mathscr{X}$ is equicontinuous; and then, $\mathscr{X}$ is distal for $\texttt{P}(\mathscr{X})=\Delta_X$. Let $\varepsilon\in\mathscr{U}_X$ and $\beta\in\mathscr{U}_X$ with $\beta^3\subset\varepsilon$. Then there exists an index $\delta\in\mathscr{U}_X$ with $\delta\subset\beta$ such that $T\delta\subset\beta$. Since $X$ is compact, there exists a finite set $\{x_1,\dotsc,x_n\}$ in $X$ such that $X=\delta[x_1]\cup\dotsm\cup\delta[x_n]$. As $(x_1,\dotsc,x_n)$ is an a.p. point in $X^n$, there is a syndetic subset $A$ of $T$ such that $A(x_1,\dotsc,x_n)\subseteq\delta[x_1]\times\dotsm\times\delta[x_n]$. Now for $y\in X$ we have for some $i$ with $1\le i\le n$ that
$(y,x_i)\in\delta$, $A(y,x_i)\subset\beta$ and $Ax_i\subset\delta[x_i]$.
Thus $Ay\subset\beta^3[y]$. Hence $Ay\subset\varepsilon[y]$ for all $y\in X$. The proof is complete.
\end{proof}
\subsubsection{Thick stability}\label{s2.2.2}
Uniform thick stability was investigated in \cite{Ba, C1} and \cite[$\S$1.6]{B79} for compact flows. Here we will now consider the nonuniform noncompact case.

\begin{lem}\label{l2.4}
Let $\mathscr{X}$ be a minimal flow and $S$ a$\{$ thick$\}$$\{$syndetic$\}$ subset of $T$. If $\mathscr{X}$ is Lyapunov $S$-stable at some point $x\in X$ w.r.t. $X$, then $\mathscr{X}$ is $\{$thickly$\}$$\{$syndetically$\}$ stable.
\end{lem}

\begin{proof}
Obvious.
\end{proof}

\begin{thm}\label{t2.5}
Let $\mathscr{X}$ be a thickly stable flow with $T$ an abelian group. If $x\in X$ is an a.p. point under $\mathscr{X}$, then $(T,\overline{Tx})$ is an a.p. subflow of $\mathscr{X}$ with discrete phase group $T$, and moreover, $(T,\overline{Tx})$ is an equicontinuous subflow of $\mathscr{X}$.
\end{thm}

\begin{proof}
Since $X$ is a regular space so that the orbit closure of an a.p. point is minimal (cf.~\cite[Thm.~1.5.3]{B79}), we can assume $\mathscr{X}=(T,\overline{Tx})$ is minimal without loss of generality. Moreover, we may assume $x$ is a.p. under $\mathscr{X}$ with $T$ a discrete group.

Let $\varepsilon,\alpha\in\mathscr{U}_X$ with $\alpha^2\subseteq\varepsilon$. Since $\mathscr{X}$ is thickly stable at $x$ by hypothesis, there exists an index $\gamma\in\mathscr{U}_X$ and a thick subset $S$ of $T$ such that $s(\gamma[x])\subseteq\alpha[sx]$ for all $s\in S$. Take an index $\delta\in\mathscr{U}_X$ with $\delta^2\subseteq\gamma$. Since $x$ is a.p. under $\mathscr{X}$ and $T$ is a topological group, there exists a discretely syndetic subset $A$ of $T$ such that $\overline{Ax}\subseteq\delta[x]$. To show that $\mathscr{X}$ is a.p., it suffices to prove that $Atx\subseteq\varepsilon[tx]$ for all $t\in T$.
For this, let $t\in T$, $a\in A$. There exists an index $\sigma\in\mathscr{U}_X$ with $\sigma\subset\delta$ such that $a(\sigma[x])\subseteq\delta[ax]$. Then there is a syndetic subset $B$ of $T$ such that $Bx\subseteq\sigma[x]$. By $T=BS$, there are elements $b\in B$ and $s\in S$ such that $b^{-1}t=s$. By $bx\in Bx\subseteq\sigma[x]$, it follows that $abx\in\delta[ax]\subseteq\gamma[x]$. By $b^{-1}t=s\in S$ and commutativity of $T$, it follows that
$atx=(b^{-1}t)(ab)x\in\alpha[b^{-1}tx]$. Moreover, $b\in B$ implies $bx\in\delta[x]\subseteq\gamma[x]$. Hence $tx=(b^{-1}t)bx\in\alpha[b^{-1}tx]$. Therefore, $(atx,tx)\in\alpha^2\subseteq\varepsilon$ for all $t\in T$ and all $a\in A$.

Next we shall prove that $\mathscr{X}$ is equicontinuous. Let $\varepsilon\in\mathscr{U}_X$ there exists an index $\delta_1\in\mathscr{U}_X$ and a syndetic set $A\subseteq T$ such that $A\delta_1\subseteq\varepsilon$. Select a compact set $K\subseteq T$ such that $KA=T=AK$ (for $T$ is abelian). Given $x_0\in X$, there exists an $\delta\in\mathscr{U}_X$ such that $k\delta[x_0]\subseteq\delta_1[kx_0]$ for all $k\in K$. Thus, $t\delta[x_0]\subseteq\varepsilon[tx_0]$ for all $t\in T$ and then $\mathscr{X}$ is equicontinuous.
The proof is complete.
\end{proof}

\begin{t2.5a}
If $\mathscr{X}$ is a locally a.p. thickly stable compact flow with $T$ abelian, then $\mathscr{X}$ is an a.p. flow.
\end{t2.5a}

\begin{proof}
Since $\mathscr{X}$ is locally a.p., $\mathscr{X}$ is pointwise a.p. and $\texttt{P}(\mathscr{X})=\texttt{RP}(\mathscr{X})$. Then $\texttt{RP}(\mathscr{X})=\Delta_X$ by Theorem~\ref{t2.5}. This together with Lemma~\ref{l2.3} proves Corollary~\ref{t2.5}a.
\end{proof}

\begin{t2.5b}[{cf.~\cite[Thm.~2]{E} or \cite[Thm.~2.8.8]{B79}}]
Let $\mathscr{X}$ be a flow with $T$ abelian. Let $S$ be a thick subset of $T$ and $x\in X$ an a.p. point under $\mathscr{X}$. If $\mathscr{X}$ is Lyapunov $S$-stable at $x$ w.r.t. $Tx$, then $x$ is a Bohr a.p. point under $\mathscr{X}$.
\end{t2.5b}

\begin{proof}
By Lemma~\ref{l2.4} and Theorem~\ref{t2.5}.
\end{proof}

When $T$ is not necessarily an abelian topological group, we can conclude the following result using ``$x$ regularly a.p.'' instead of ``$x$ a.p.'' under $\mathscr{X}$:

\begin{thm}\label{t2.6}
Let $\mathscr{X}$ be a flow with $X$ a Baire space. If $\mathscr{X}$ is thickly stable and it is regularly a.p. at some point $x\in X$, then $(T,\overline{Tx})$ is a pointwise regularly a.p. and equicontinuous subflow of $\mathscr{X}$.
\end{thm}

\begin{proof}
We can assume $X=\overline{Tx}$ without loss of generality for $\overline{Tx}$ is a minimal subset of $X$. Let $\varepsilon,\alpha\in\mathscr{U}_X$ with $\alpha^2\subseteq\varepsilon$. There exists an index $\gamma\in\mathscr{U}_X$ and a thick subset $S$ of $T$ such that $s\gamma[x]\subseteq\alpha[sx]$ for all $s\in S$. Take an index $\delta\in\mathscr{U}_X$ with $\delta^2\subseteq\gamma$. Since $x$ is a.p. under $\mathscr{X}$, there is a discretely syndetic set $A$ in $T$ with $\overline{Ax}\subseteq\delta[x]$ and $\texttt{int}\,\overline{Ax}\not=\emptyset$. To show $\mathscr{X}$ is Lyapunov $T$-stable at $x$ w.r.t. $X$, it suffices to prove that $tax\in\varepsilon[tx]$ for all $t\in T$ and every $a\in A$.

For that, let $t\in T$ and $a\in A$ be arbitrarily given. Select an index $\sigma\in\mathscr{U}_X$ with $a\sigma[x]\subseteq\delta[ax]$. Since $x$ is regularly a.p. under $\mathscr{X}$, there exists a syndetic normal closed subgroup $B$ of $T$ such that $Bx\subseteq\sigma[x]$. Thus $aBx\subseteq\delta[ax]\subseteq\gamma[x]$. Since $S$ is thick and $tB$ is syndetic in $T$, $tB\cap S\not=\emptyset$. Thus $tb=s$ for some $b\in B$ and $s\in S$. Then
$$
tax=(tb)(b^{-1}a)x\in(tb)(aB)x\subseteq\alpha[tbx]\quad \textrm{and}\quad
tx=(tb)b^{-1}x\in(tb)Bx\subseteq(tb)\gamma[x]\subseteq\alpha[tbx].
$$
Whence $(tax,tx)\in\alpha^2\subseteq\varepsilon$ for all $t\in T$ and all $a\in A$.

Further by minimality of $X$, $\mathscr{X}$ is Lyapunov $T$-stable. On the other hand, $\mathscr{X}$ is regularly a.p. at each point of $Tx$. Then by Lyapunov stability, each point of $X=\overline{Tx}$ is regularly a.p. under $\mathscr{X}$. The proof is completed.
\end{proof}

\begin{t2.6a}
If $\mathscr{X}$ is a point-regularly a.p. thickly stable compact flow, then $\mathscr{X}$ is a regularly a.p. flow.
\end{t2.6a}

\begin{proof}
Use Theorem~\ref{t2.6} and the uniform equicontinuity of $\mathscr{X}$.
\end{proof}

\begin{t2.6b}
Let $\mathscr{X}$ be a flow with $X$ a Baire space. Let $S$ be a thick subset of $T$ and $x\in X$ a regularly a.p. point under $\mathscr{X}$. If $\mathscr{X}$ is Lyapunov $S$-stable at $x$ w.r.t. $Tx$, then $\mathscr{X}$ is Lyapunov $T$-stable at $x$ w.r.t. $\overline{Tx}$.
\end{t2.6b}

\begin{proof}
By Lemma~\ref{l2.4} and Theorem~\ref{t2.6}.
\end{proof}

Note that a point-regularly a.p. compact flow need not be a.p.; for example, \cite[Thm.~12.55]{GH}. So condition ``$\mathscr{X}$ is thickly stable'' is critical for Theorem~\ref{t2.6}.
\subsubsection{Lyapunov thick-subsemigroup stability}
\begin{se}[$P$-limit set]\label{2.7}
Let $x\in X$, and $P\subseteq T$. The \textit{$P$-limit set of $x$}, denoted $P_{\!x}$, is defined to be $\bigcap_{t\in T}\overline{Ptx}$; that is, $y\in P_{\!x}$ iff for every $t\in T$ there is a net $p_j\in P$ such that $p_jtx\to y$. Each point of $P_{\!x}$ is called a \textit{$P$-limit point of $x$} under $\mathscr{X}$. See \cite[Def.~6.33]{GH} or \cite[Def.~2.8.9]{B79}.
\begin{enumerate}[$\bullet$]
\item If $x\in X$ is a.p. under $\mathscr{X}$ and $P$ a thick set in $T$, then $x\in P_{\!x}$.
\end{enumerate}
Note here that $P_x$ is different from the proximal cell $\texttt{P}[x]$ under $\mathscr{X}$.
\end{se}

\begin{2.7a}[{cf.~\cite[Thm.~6.07]{GH} for $T$ a generative group}]
If $P$ is a thick semigroup in $T$ and $K$ a compact subset of $T$ with $e\in K$, then $\bigcap_{k\in K}kP$ is a thick semigroup in $T$.
\end{2.7a}

\begin{proof}
Write $Q=\bigcap_{k\in K}kP$. Since $P\cap kP$ is a semigroup for all $k\in K$, $Q$ is a semigroup in $T$. Let $C$ be a compact subset of $T$ and set $D=C\cup K^{-1}C$. Since $D$ is compact, we can select $t\in T$ such that $Dt\subseteq P$. Now we have for $k\in K$ that $C\subseteq D\cap kD$ and $Ct\subseteq Dt\cap kDt\subseteq P\cap kP$. Thus $Ct\subseteq Q$ and $Q$ is thick in $T$. The proof is completed.
\end{proof}

\begin{2.7b}[{cf.~\cite[Lem.~2.8.10]{B79} for $T$ an abelian group}]
Let $P$ be a normal thick subsemigroup of $T$. Let $x\in X$. Then based on $\mathscr{X}$:
\begin{enumerate}[$(1)$]
\item $P_{\!x}$ is closed and $T$-invariant.
\item If $\overline{Px}$ is compact, then $P_{\!x}\not=\emptyset$.
\item $P_{\!x}=\bigcap_{t\in P}\overline{Ptx}$.
\end{enumerate}
Here $P$ `normal' means `$Pt=tP\ \forall t\in T$'.
\end{2.7b}

\begin{proof}
\item (1): Clearly, $P_{\!x}$ is closed. Let $y\in P_{\!x}$ and $s,t\in T$. Since $y\in\overline{Ps^{-1}tx}$, so $sy\in s\overline{Ps^{-1}tx}=\overline{Ptx}$ and $sy\in P_{\!x}$. Thus $P_{\!x}$ is $T$-invariant.

\item (2): Suppose $\overline{Px}$ is compact. Let $t_1,\dotsc,t_n\in T$. By Lemma~\ref{2.7}a, $P\cap t_1P\cap\dotsm\cap t_nP$ is a thick semigroup in $T$. Thus
${\bigcap}_{i=1}^n\overline{Pt_ix}={\bigcap}_{i=1}^n\overline{t_iPx}\supseteq\left({\bigcap}_{i=1}^n t_iP\right)x\not=\emptyset$
and $P_{\!x}\not=\emptyset$.

\item (3): Since $P$ is normal thick in $T$, $P^{-1}$ is also thick in $T$. Then for all $t\in T$, there is some $s_t\in P$ with $s_tt\in P$. So
$P_{\!x}\subseteq{\bigcap}_{p\in P}\overline{Ppx}\subseteq{\bigcap}_{t\in T}\overline{Ps_ttx}\subseteq{\bigcap}_{t\in T}\overline{Ptx}\subseteq P_{\!x}$.
The proof is complete.
\end{proof}

\begin{lem}[{a special case of \cite[Thm.~1.15]{AD}}]\label{l2.8}
Let $\mathscr{X}$ be a compact flow and $S$ a subsemigroup of $T$. If $\mathscr{X}$ is Lyapunov $S$-stable, then it is Lyapunov $\langle S\rangle$-stable. Here $\langle S\rangle$ is the subgroup of $T$ generated by $S$.
\end{lem}

\begin{thm}\label{t2.9}
Let $\mathscr{X}$ be a compact flow and $S$ a thick subsemigroup of $T$. If $\mathscr{X}$ is Lyapunov $S$-stable, then $\mathscr{X}$ is an a.p. flow.
\end{thm}

\begin{proof}
By Lemma~\ref{l2.8}, $(\langle S\rangle,X)$ is an equicontinuous flow. Since $S$ is thick in $T$, so is $\langle S\rangle$. However, every thick subgroup of $T$ coincides with $T$. Thus $\langle S\rangle=T$ and $\mathscr{X}$ is equicontinuous. Since $X$ is compact, $\mathscr{X}$ is a.p. by Lemma~\ref{l2.3}. The proof is complete.
\end{proof}

\begin{t2.9a}[{cf.~\cite{M} \cite[Thm.~V.8.12]{NS} for $T=\mathbb{R}$ and $S=\mathbb{R}_+$; \cite[Thm.~3]{E} \cite[Thm.~2.8.12]{B79} for $T$ abelian}]
Let $\mathscr{X}$ be a flow and let $S$ be a thick normal subsemigroup of $T$. Let $x\in X$ such that $\overline{S^{-1}x}$ is compact and
that $(T,Tx)$ is a uniformly Lyapunov $S$-stable subflow of $\mathscr{X}$.
Then $(T,\overline{Tx})$ is a compact minimal a.p. subflow of $\mathscr{X}$.
\end{t2.9a}

\begin{proof}
Let $P=S^{-1}$; then $P\triangleleft T$ and $P$ is thick in $T$. By Lemma~\ref{2.7}b, $P_{\!x}$ is an $T$-invariant closed nonempty compact subset of $X$. Then we can take a $T$-minimal set $M$ with $M\subseteq P_{\!x}\subseteq \overline{Tx}$. Since $(T,Tx)$ is uniformly Lyapunov $S$-stable, $(T,\overline{Tx})$ is also uniformly Lyapunov $S$-stable.
Let $y\in M$. Since $M\subseteq \overline{S^{-1}x}$, there is a net $s_n\in S$ such that $s_n^{-1}x\to y$ in $\overline{Tx}$. By uniform $S$-stability, $s_ny\to x\in M$. Thus $\overline{Tx}=M$ is compact. Then by Theorem~\ref{t2.9}, $(T,\overline{Tx})$ is equicontinuous and so an a.p. minimal compact subflow of $\mathscr{X}$. The proof is completed.
\end{proof}

\begin{t2.9b}[{cf.~\cite[Thm.~4]{E} \cite[Thm.~2.8.14]{B79} for $T$ abelian}]
Let $\mathscr{X}$ be a flow and let $P$ be a thick normal subsemigroup of $T$. Let $x\in X$. Then following statements are equivalent:
\begin{enumerate}[$(1)$]
\item $x$ is Bohr a.p. under $\mathscr{X}$ with $\overline{Tx}$ compact.
\item $x\in P_{\!x}$, $P_{\!x}$ is compact, and $\mathscr{X}$ is uniformly Lyapunov $P$-stable at $Tx$ w.r.t. $P_{\!x}$.
\end{enumerate}
\end{t2.9b}

\begin{proof}
Assume (1). Clearly, $P_{\!x}=\overline{Tx}$ is compact minimal by Lemma~\ref{2.7}b, and $(T,\overline{Tx})$ is equicontinuous. Thus $x\in P_{\!x}$ and $Tx$ is uniformly Lyapunov $P$-stable w.r.t. $P_{\!x}$. Then (1) implies (2). Conversely assume (2). By Lemma~\ref{2.7}b, $\overline{Tx}\subseteq P_{\!x}$ and $\overline{Tx}= P_{\!x}$. Now applying Theorem~\ref{t2.9} with $\overline{Tx}$ and $P$ instead of $X$ and $S$ respectively, $(T,P_{\!x})$ is a.p. Since $x\in P_{\!x}$, so $x$ is Bohr a.p. and (2) implies (1). The proof is complete.
\end{proof}

\begin{t2.9c}[{cf.~\cite[Thm.~5]{E} for $T$ abelian}]
Let $\mathscr{X}$ be a flow and let $P$ be a thick normal subsemigroup of $T$. Suppose $y\in X$ such that:
\begin{enumerate}[$(1)$]
\item  The $P$-limit set of $y$, $P_{\!y}$, is compact nonempty with $P_{\!x}=P_{\!y}$ for all $x\in P_{\!y}$.
\item  $\mathscr{X}$ is Lyapunov $P$-stable at $\overline{Ty}$ w.r.t. $P_{\!y}$.
\end{enumerate}
Then $P_{\!y}$ is the closure of a Bohr a.p. point of $\mathscr{X}$.
\end{t2.9c}

\begin{proof}
Condition (1) implies that $P_{\!y}$ is a minimal set for $\mathscr{X}$. Since $P_{\!y}$ is compact, condition (2) implies that $(T,P_{\!y})$ is uniformly Lyapunov $P$-stable. Then by Theorem~\ref{t2.9}, $(T,P_{\!y})$ is (Bohr) a.p. so that $P_{\!y}$ is the closure of a Bohr a.p. orbit of $\mathscr{X}$. The proof is complete.
\end{proof}

\subsubsection{Syndetic distality}
It is easy to verify that every uniformly distal compact flow is a.p. (cf.~\cite[Lem.~1.7(2)]{AD} and \cite[Thm.~1.1.5]{CD20}). We now improve this result as follows:

\begin{thm}\label{t2.10}
If $\mathscr{X}$ is a syndetically distal compact flow, then it is pointwise Bohr a.p.
\end{thm}

\begin{proof}
First we claim that $\mathscr{X}$ is distal. Indeed, let $x\in X$ and $y\in\texttt{P}[x]$ with $x\not=y$. Then there is an index $\epsilon\in\mathscr{U}_X$ with $(x,y)\notin\epsilon$. Since $\mathscr{X}$ is syndetically distal, there exists an index $\gamma\in\mathscr{U}_X$ and a syndetic subset $A$ of $T$ such that $t^{-1}(\gamma[tx])\subseteq\epsilon[x]$ for all $t\in A$. However, by $y\in\texttt{P}[x]$, $B:=\{t\in T\,|\,(tx,ty)\in\gamma\}$ is a thick subset of $T$. By $A\cap B\not=\emptyset$, take $t_0\in A\cap B$. Then
$t_0y\in\gamma[t_0x]$ and $t_0^{-1}(\gamma[t_0x])\subseteq\epsilon[x]$. So $(x,y)\in\epsilon$, a contradiction. Thus $\texttt{P}(\mathscr{X})=\Delta_X$ and $\mathscr{X}$ is a distal compact flow.

So to show $\mathscr{X}$ pointwise Bohr a.p., we can assume $\mathscr{X}$ is minimal distal, and we shall show $\mathscr{X}$ is a.p. For this, define the strong regionally proximal relation on $X$ by
\begin{equation*}
\texttt{U}(\mathscr{X})=\{(x,y)\,|\,\exists\, y_n\in X\to y\textrm{ and }t_n\in T\textrm{ s.t. }(t_nx,t_ny_n)\to\Delta_X\}
\end{equation*}
and let $\texttt{U}[x]=\{y\in X\,|\,(x,y)\in \texttt{U}(\mathscr{X})\}$ for all $x\in X$.
Then
$\texttt{U}[x]={\bigcap}_{\alpha\in\mathscr{U}_X}\overline{{\bigcup}_{t\in T}t^{-1}(\alpha[tx])}$ for all $x\in X$.
Since $\mathscr{X}$ is minimal distal, $\mathscr{X}$ admits an invariant Borel probability measure (cf.~\cite{F63}).
Then it follows from \cite{Mc} that $\texttt{RP}(\mathscr{X})=\texttt{U}(\mathscr{X})$. Therefore,
$$
\texttt{RP}[x]=\texttt{U}[x]={\bigcap}_{\alpha\in\mathscr{U}_X}\overline{{\bigcup}_{t\in T}t^{-1}(\alpha[tx])}\quad \forall x\in X.
$$
Let $x\in X$ and $\varepsilon\in\mathscr{U}_X$. There exists an index $\delta\in\mathscr{U}_X$ and a syndetic subset $S$ of $T$ such that
$s^{-1}(\delta[sx])\subseteq\varepsilon[x]$ for all $s\in S$. Since $S$ is syndetic, there is a compact subset $K$ of $T$ such that $T=KS$. Further, since $X$ is compact and the phase mapping $T\times X\rightarrow X$ is jointly continuous, there exists an index $\alpha\in\mathscr{U}_X$ such that $K^{-1}\alpha\subseteq\delta$. Let $t\in T$ be arbitrary. Then we can write $t=ks$ for some $k\in K$ and some $s\in S$. Thus
$t^{-1}(\alpha[tx])=s^{-1}(k^{-1}(\alpha[ksx]))\subseteq s^{-1}(\delta[sx])\subseteq\varepsilon[x]$.
This shows that $\texttt{RP}[x]\subseteq\overline{\varepsilon[x]}$. Since $\varepsilon$ is arbitrary, so $\texttt{RP}[x]=\{x\}$ and $\texttt{RP}(\mathscr{X})=\Delta_X$. Then $\mathscr{X}$ is equicontinuous and a.p. by Lemma~\ref{l2.3}. The proof is complete.
\end{proof}

\begin{t2.10a}
If $\mathscr{X}$ is a $\{$syndetically$\}$$\{$thickly$\}$ distal and syndetically stable compact flow, then it is an a.p. flow.
\end{t2.10a}

\begin{proof}
Since $\mathscr{X}$ is syndetically stable, it follows that $\texttt{P}(\mathscr{X})=\texttt{RP}(\mathscr{X})$. If $\mathscr{X}$ is syndetically distal, then by Theorem~\ref{t2.10}, $\texttt{RP}(\mathscr{X})=\Delta_X$. If $\mathscr{X}$ is thickly distal, then by \ref{s2}.0h $\texttt{RP}(\mathscr{X})=\Delta_X$. So $\mathscr{X}$ is an a.p. flow by Lemma~\ref{l2.3}. The proof is complete.
\end{proof}

\begin{t2.10b}
If $\mathscr{X}$ is a syndetically distal T.T. compact flow, then it is an a.p. minimal flow.
\end{t2.10b}
\begin{proof}
By Theorem~\ref{t2.10} and (\ref{s1}.7b).
\end{proof}

Recall that $\mathscr{X}$ is said to be uniformly syndetically stable/equicontinuous if for every index $\varepsilon\in\mathscr{U}_X$ there exists an index $\delta\in\mathscr{U}_X$ and a syndetic subset $S$ of $T$ such that $S\delta\subseteq\varepsilon$.

\begin{t2.10c}[{\cite{C}}]
If $\mathscr{X}$ is a uniformly syndetically stable compact flow with $T$ an abelian group, then $\mathscr{X}$ is an a.p. flow.
\end{t2.10c}

\begin{proof}
Since $T$ is abelian, $\mathscr{X}$ is syndetically distal. Then $\mathscr{X}$ is a.p. by Corollary~\ref{t2.10}a.
\end{proof}

\begin{t2.10d}
If $\mathscr{X}$ is a uniformly syndetically stable compact flow with $T$ non-abelian, is $\mathscr{X}$ an a.p. flow?
\end{t2.10d}

\noindent We say that
\begin{enumerate}[$\bullet$]
\item $\mathscr{X}$ is \textit{uniformly syndetically regularly stable} if for every index $\varepsilon\in\mathscr{U}_X$ there exists an index $\delta\in\mathscr{U}_X$ and a syndetic subsemigroup $A$ of $T$ such that $A\delta\subset\varepsilon$.

\item $\mathscr{X}$ is \textit{uniformly left-syndetically stable} if for every index $\varepsilon\in\mathscr{U}_X$ there exists an index $\delta\in\mathscr{U}_X$ and a ``left-syndetic'' subset $A$ of $T$ such that $A\delta\subset\varepsilon$.

\item $\mathscr{X}$ is \textit{uniformly left-thickly stable} if for every index $\varepsilon\in\mathscr{U}_X$ there exists an index $\delta\in\mathscr{U}_X$ and a ``left-thick'' subset $A$ of $T$ such that $A\delta\subset\varepsilon$.
\end{enumerate}
Here $A$ is \textit{left-syndetic} iff there exists a compact subset $K$ of $T$ such that
    $A\cap tK\not=\emptyset$ for all $t\in T$, and, $A$ is \textit{left-thick} iff for every compact subset $K$ of $T$ there exists an element $t\in T$ such that
    $tK\subseteq A$.

\begin{t2.10e}
If $\mathscr{X}$ is a uniformly syndetically regularly stable compact flow, then it is an a.p. flow.
\end{t2.10e}

\begin{proof}
Let $\varepsilon\in\mathscr{U}_X$. We can take an index $\alpha\in\mathscr{U}_X$ such that $\bar{\alpha}\subseteq\varepsilon$. Further there exists an index $\delta\in\mathscr{U}_X$ and a syndetic subsemigroup $A$ of $T$ such that $A\delta\subseteq\alpha$. Let $S=\overline{A}$. Then $S$ is a closed syndetic subsemigroup of $T$ such that $S\delta\subseteq\varepsilon$. Noting that $S$ is a subgroup of $T$ (cf.~\cite[Lem.~2.8.17]{B79}), $\mathscr{X}$ is syndetically distal and syndetically stable. Thus $\mathscr{X}$ is a.p. by Corollary~\ref{t2.10}a. The proof is complete.
\end{proof}

\begin{t2.10f}
If $\mathscr{X}$ is a uniformly left-syndetically stable compact flow, then it is an a.p. flow.
\end{t2.10f}

\begin{proof}
Let $\varepsilon\in\mathscr{U}_X$. Then there exists an index $\delta\in\mathscr{U}_X$ and a left-syndetic subset $A$ of $T$ such that $A\delta\subset\varepsilon$. Since $A$ is left-syndetic, there is a compact subset $K$ of $T$ such that $T=AK$. Further, there exists an index $\alpha\in\mathscr{U}_X$ such that $K\alpha\subseteq\delta$. Thus $T\alpha=AK\alpha\subseteq A\delta\subset\varepsilon$. Hence $\mathscr{X}$ is an equicontinuous and a.p. flow. The proof is complete.
\end{proof}

\begin{t2.10g}
If $\mathscr{X}$ is a uniformly syndetically distal compact flow, then it is an a.p. flow.
\end{t2.10g}

\begin{proof}
Let $\mathscr{X}$ be uniformly syndetically distal. Since the inverse of a syndetic set is left-syndetic, $\mathscr{X}$ is clearly a uniformly left-syndetic stable flow. Thus $\mathscr{X}$ is a.p. by Theorem~\ref{t2.10}f. The proof is complete.
\end{proof}

Theorem~\ref{t2.10}f is comparable with Clay's theorem (\cite[Thm.~14]{C1} \cite[Thm.~1.6.21]{B79}): $\mathscr{X}$ is uniformly left-thickly stable iff it is equicontinuous. Moreover, following \cite[Thm.~8.3]{AD}, $\mathscr{X}$ is equicontinuous iff for every $\alpha\in\mathscr{U}_X$ there exists a syndetic subset $A$ of $T$ such that $A=A^{-1}$ (so $A$ is left-syndetic) and $Ax\subseteq\alpha[x]$ for all $x\in X$.

\section{Canonical M-dynamics of non-PI extensions}\label{3}
Let $\pi\colon\mathscr{X}\rightarrow\mathscr{Y}$ be a nontrivial extension of minimal compact dynamics, unless specified otherwise, in this section.

Recall that $\pi$ is weakly mixing if $\mathscr{R}_\pi$ is T.T.; and $\pi$ is ``relatively incontractible'' (RIC) if $\pi$ is open such that
$\texttt{R}_\pi^n$ has a dense set of a.p. points for all $n\ge2$ (cf.~\cite{CD20}).
Here we shall define canonical M-dynamics in $\texttt{R}_\pi$, which will be useful for describing the $\pi$-relative unpredicted dynamics of $\mathscr{X}$. Moreover, we will characterize intrinsically PI-extensions of minimal dynamics in this section.

In the sequel, $T$ is thought of as a discrete space and let $\mathbb{I}$ be a minimal left ideal in $\beta T$. Let $\mathscr{I}=(T,\mathbb{I})$ as in $\S$\ref{s2.1}. Let $\pi_X\colon \mathscr{I}\rightarrow \mathscr{X}$ and $\pi_Y\colon\mathscr{I}\rightarrow \mathscr{Y}$ be two extensions such that $\pi_Y=\pi\circ\pi_X$. Associated to $\pi_X$ and $\pi_Y$ we can define two groups in $\textrm{Aut}\,(\mathscr{I})$:

\begin{se}\label{3.1}
$\mathbb{A}=\{\alpha\in \textrm{Aut}\,(\mathscr{I})\,|\,\pi_X=\pi_X\circ\alpha\}$ and $\mathbb{F}=\{\alpha\in \textrm{Aut}\,(\mathscr{I})\,|\,\pi_Y=\pi_Y\circ\alpha\}$.
\end{se}
\noindent
Then $\mathbb{A}$ and $\mathbb{F}$ are $\tau$-closed subgroups of $\textrm{Aut}\,(\mathscr{I})$ with $\mathbb{A}\subseteq\mathbb{F}$, and, $\mathbb{A}=\mathbb{F}$ iff $\pi$ is proximal; see \cite[Appendix~A]{CD20}.
Put $\mathbb{F}_0=\mathbb{F}$, $\mathbb{F}_1=\mathbb{F}^\prime$, $\mathbb{F}_2=\mathbb{F}_1^\prime$, and for every ordinal $o$, $\mathbb{F}_{o+1}=\mathbb{F}_o^\prime$. If $o$ is a limit ordinal, put $\mathbb{F}_o=\bigcap_{\gamma<o}\mathbb{F}_\gamma$. Then by transfinite induction there there exists a least ordinal $o$ such that
$\mathbb{F}_0\supsetneq\mathbb{F}_1\supsetneq\mathbb{F}_2\supsetneq\dotsm\supsetneq\mathbb{F}_o=\mathbb{F}_{o+1}=\dotsm$.
Write
\begin{se}\label{3.2}
$\mathbb{F}_\infty=\mathbb{F}_o$,
which is called the \textit{core} of $\mathbb{F}$ and which is a $\tau$-closed subgroup of $\mathbb{F}$ such that $\mathbb{A}\mathbb{F}_\infty=\mathbb{F}_\infty\mathbb{A}$.
\end{se}
\noindent
Then:
\begin{se}\label{3.3}
$\pi$ is a \textit{PI-extension} iff $\mathbb{F}_\infty\subseteq\mathbb{A}$ (cf.~\cite[Thm.~X.4.2]{G76} or \cite[Thm.~14.23]{A88}; following from \ref{3.6}).
\end{se}
\noindent
Notice that the PI condition of $\pi$ is independent of the choice of $\pi_X$ and $\pi_Y$. So PI-extension is an intrinsic concept. Every PI-extension has ``good'' structure; see \ref{3.6} below.

There are many intrinsic characterizations of PI-extensions; see \cite{B77, MN, Wo}. In particular, ``$\pi$ is PI iff every M-dynamic in $\texttt{R}_\pi$ is minimal'' (due to Bronstein, cf.~\cite{B77, Wo, Ge01}). In fact, we will show that $\pi$ is PI iff every M-dynamic containing $\Delta_X$ in $\texttt{R}_\pi$ is just equal to $\Delta_X$ (see Theorem~\ref{3.10} below). So if $\pi$ is not PI, then $\texttt{R}_\pi$ contains a non-minimal M-dynamic.

However, we shall be only interested to the non-minimal M-dynamics that contains $\Delta_X$. This type of M-dynamic is useful for us to define chaos of non-PI extensions. For this, we now put
\begin{se}\label{3.4}
$\texttt{M}_\infty=(\pi_X\times\pi_X)\overline{\Gamma[\mathbb{F}_\infty]}$, where $\Gamma[\mathbb{F}_\infty]={\bigcup}_{\alpha\in\mathbb{F}_\infty}\Gamma[\alpha]$.
\end{se}
\noindent
Clearly, $\texttt{M}_\infty$ is a closed invariant symmetric subset of $\texttt{R}_\pi$ containing a dense set of a.p. points of $\texttt{RP}_\pi$ (cf.~\cite[Lem.~A.4.4]{CD20}).

Although we have not shown that $\texttt{M}_\infty$ is an AP-transitive relation on $X$ here, yet we can first obtain the following:

\begin{thm}\label{3.5}
If $\pi$ is not PI, then $(T,\texttt{M}_\infty)$ is an M-dynamic in $\texttt{R}_\pi$ with $\Delta_X\varsubsetneq \texttt{M}_\infty$ and $\texttt{RP}_{\pi|\texttt{M}_\infty}=\texttt{M}_\infty$.
\end{thm}

\begin{proof}
Since $\pi$ is not PI, $\mathbb{F}_\infty\nsubseteq\mathbb{A}$ by \ref{3.3} and $\textrm{id}_\I\in\mathbb{F}_\infty\cap\mathbb{A}$. Thus, $\Delta_X\varsubsetneq \texttt{M}_\infty$. Let $U$ and $V$ be two nonempty open sets in $\texttt{M}_\infty$. There exist two elements $\alpha,\gamma\in\mathbb{F}_\infty$ such that $U\cap(\pi_X\times\pi_X)\Gamma[\alpha]\not=\emptyset$ and $V\cap(\pi_X\times\pi_X)\Gamma[\gamma]\not=\emptyset$. As $\mathbb{F}_\infty=\mathbb{F}_\infty^\prime$, it follows that $(T,\texttt{M}_\infty)$ is T.T. (cf.~Proof of Theorem~\ref{thm2}) so that $(T,\texttt{M}_\infty)$ is an M-dynamic and then $\texttt{RP}_{\pi|\texttt{M}_\infty}=\texttt{M}_\infty$. The proof is complete.
\end{proof}

Therefore, in view of Bronstein's theorem and Theorem~\ref{3.5}, if $\texttt{R}_\pi$ has a non-minimal M-subdynamic, then $(T,\texttt{M}_\infty)$ is an M-dynamic in $\texttt{R}_\pi$ with $\Delta_X\varsubsetneq \texttt{M}_\infty$.

\begin{se}[Ellis-Glasner-Shapiro tower]\label{3.6}
Based on the classical Ellis-Glasner-Shapiro structure theorem (cf.~\cite{EGS, G76, V77, A88} and \cite{Ge01}), for $\pi\colon\mathscr{X}\rightarrow\mathscr{Y}$, there exists the following canonical CD of minimal compact dynamics and $T$-extensions:
$$
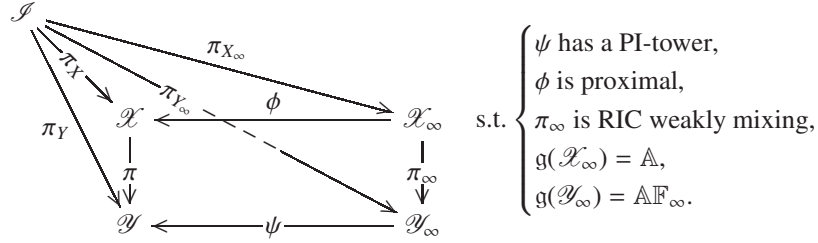
\begin{diagram}
\mathscr{I}&&&&\\
&\rdTo~{\pi_X}\rdTo(4,2)^{\pi_{X_\infty}}\rdTo(2,4)_{\pi_Y}\rdTo(4,4)~{\pi_{Y_\infty}----}\\
&&\mathscr{X}&\lTo^{\quad\quad\qquad\phi\qquad\quad\quad}&\mathscr{X}_\infty\\
&&\dTo~{\pi}&&\dTo~{\pi_\infty}\\
&&\mathscr{Y}&\lTo~\psi&\mathscr{Y}_\infty
\end{diagram}
\qquad \textrm{s.t.}
\begin{cases}\psi\textrm{ has a PI-tower},\\\phi\textrm{ is proximal},\\\pi_\infty\textrm{ is RIC weakly mixing},\\ \mathfrak{g}(\mathscr{X}_\infty)=\mathbb{A},\\
\mathfrak{g}(\mathscr{Y}_\infty)=\mathbb{A}\mathbb{F}_\infty.
\end{cases}
$$
Here $\mathbb{A}, \mathbb{F}$ are as in Def.~\ref{3.1} and
\begin{gather*}
\mathfrak{g}(\mathscr{X}_\infty)=\{\alpha\in\textrm{Aut}\,(\mathscr{I})\,|\,\pi_{X_\infty}=\pi_{X_\infty}\circ\alpha\},
\quad \mathfrak{g}(\mathscr{Y}_\infty)=\{\alpha\in\textrm{Aut}\,(\mathscr{I})\,|\,\pi_{Y_\infty}=\pi_{Y_\infty}\circ\alpha\}.
\end{gather*}
Notice that $(\phi\times\phi)\texttt{R}_{\pi_\infty}\varsubsetneq \texttt{R}_\pi$ in general, for $\psi$ need generally not be proximal and $\pi$ need not satisfy the Bronstein condition. Moreover, $\pi$ is PI iff $\pi_\infty$ is 1-1 onto iff $\pi_\infty$ is proximal iff $\mathbb{A}=\mathbb{A}\mathbb{F}_\infty$.
\end{se}

\begin{thm}\label{3.7}
Let $\pi\colon\mathscr{X}\rightarrow\mathscr{Y}$ be not PI. Then $\texttt{M}_\infty=(\phi\times\phi)\texttt{R}_{\pi_\infty}$ and $(\pi_X\times\pi_X)\Gamma[\mathbb{F}_\infty]$ is the set of all a.p. points of $(T,\texttt{M}_\infty)$.
\end{thm}

\begin{proof}
By Def.~\ref{3.4} and $\mathbb{F}_\infty\subseteq\mathfrak{g}(\mathscr{Y}_\infty)$, it follows that $\texttt{M}_\infty\subseteq(\phi\times\phi)\texttt{R}_{\pi_\infty}$. Let $(w,w^\prime)\in \texttt{R}_{\pi_\infty}$ be an a.p. point. Since $\pi_\infty$ satisfies the Bronstein condition and $(w,w^\prime)\in\texttt{RP}_{\pi_\infty}$, so by \cite[Thm.~3.1.3]{CD20} there is $\alpha\in(\mathbb{A}\mathbb{F}_\infty)^\prime\subseteq\mathbb{A}\mathbb{F}_\infty$ such that $(w,w^\prime)\in(\pi_{X_\infty}\times\pi_{X_\infty})\Gamma[\alpha]$. This implies that $(\phi\times\phi)(w,w^\prime)\in \texttt{M}_\infty$. Thus $\texttt{M}_\infty\supseteq(\phi\times\phi)\texttt{R}_{\pi_\infty}$ and $\texttt{M}_\infty=(\phi\times\phi)\texttt{R}_{\pi_\infty}$.

Let $(x_1,x_2)\in \texttt{M}_\infty$ be an a.p. point. By $\texttt{M}_\infty=(\phi\times\phi)\texttt{R}_{\pi_\infty}$, we can take an a.p. point $(w_1,w_2)$ in $\texttt{R}_{\pi_\infty}$ such that $\phi w_1=x_1$ and $\phi w_2=x_2$. Further, there exist elements $\alpha\in\mathbb{A}$ and $\gamma\in\mathbb{F}_\infty$ such that $(w_1,w_2)=(\pi_{X_\infty}\times\pi_{X_\infty})(m,\alpha\gamma m)=(\pi_{X_\infty}\times\pi_{X_\infty})(m,\gamma m)$ for some $m\in\I$. So by $\pi_X=\phi\circ\pi_{X_\infty}$, $(x_1,x_2)=(\pi_X\times\pi_X)(m,\gamma m)$. The proof is complete.
\end{proof}

\begin{lem}[{cf.~\cite[Lem.~2.10]{A84} for $T$ a group}]\label{3.8}
Let $\phi\colon\mathscr{X}\rightarrow\mathscr{Y}$ be a proximal extension of compact dynamics with $X$ having a dense set of a.p. points. Then $\mathscr{X}$ is an M-dynamic iff so is $\mathscr{Y}$.
\end{lem}

\begin{proof}
Necessity is obvious. Conversely, suppose $\mathscr{Y}$ is an M-dynamic. Let $K$ be an invariant closed set in $X$ with $\texttt{int}\,K\not=\emptyset$. We need to prove $K=X$.
Suppose the contrary that $K\not=X$. Then $X\setminus K\not=\emptyset$. Since a.p. points are dense in $X$, so $(X\setminus K)\cap t^{-1}\texttt{int}\,K=\emptyset$ for all $t\in T$. It is clear that $L:=\overline{X\setminus K}=X\setminus\texttt{int}\,K$ and $t(X\setminus K)\subseteq L$ for every $t\in T$. Thus $L$ is an invariant closed set in $X$ with $\texttt{int}\,L\not=\emptyset$ and $K\cup L=X$. As $\phi K\cup \phi L=\phi X=Y$, $\phi K$ or $\phi L$ must have a nonempty interior in $Y$, and so, by T.T. of $\mathscr{Y}$, $\phi K=Y$ or $\phi L=Y$.

Assume $\phi K=Y$. Let $x\in X\setminus K$ be an a.p. point. Then we can find $x^\prime\in K$ with $\phi x=\phi x^\prime$ such that $(x,x^\prime)$ is an a.p. point. Since $\phi$ is proximal, so $x=x^\prime$ and $K=X$. If $\phi L=Y$, then similarly it follows that $L=X$, contrary to $L\cap\texttt{int}\,K=\emptyset$ and $\texttt{int}\,K\not=\emptyset$.
Hence $K=X$ and $\mathscr{X}$ is T.T. and an M-dynamic. The proof is complete.
\end{proof}

\begin{lem}\label{3.9}
Let $\phi\colon\mathscr{Z}\rightarrow\mathscr{Y}$ be an extension of compact dynamics with $\mathscr{Y}$ minimal, let $\mathscr{N}$ be an $M$-dynamic in $\texttt{R}_\phi$. Then:
\begin{enumerate}[(1)]
\item If $\phi$ is proximal, then $N=\Delta_Z$.
\item If $\phi$ is a.p., then $\mathscr{N}$ is minimal.
\end{enumerate}
\end{lem}

\begin{proof}
(1) is obvious. For (2), we consider $\theta\colon\mathscr{N}\rightarrow\mathscr{Y}$, which is induced canonically by $\phi$. Then $\mathscr{N}$ is T.T. and $\theta$-distal. By (\ref{s1}.7c), $\mathscr{N}$ is minimal.
\end{proof}

\begin{thm}[{cf.~\cite{B77, Wo, Ge01} for $(1)\Leftrightarrow(2)$}]\label{3.10}
For $\pi\colon\mathscr{X}\rightarrow\mathscr{Y}$ the following are pairwise equivalent:
\begin{enumerate}[$(1)$]
\item $\pi$ is a PI-extension.
\item Every M-dynamic in $\texttt{R}_\pi$ is minimal.
\item $(T,\Delta_X)$ is the unique M-dynamic containing $\Delta_X$ in $\texttt{R}_\pi$.
\end{enumerate}
\end{thm}

\begin{proof}
\item $(1)\Rightarrow(2)$: Clearly by \ref{3.3}, $\mathscr{X}_\infty\cong \mathscr{Y}_\infty$ in the CD of \ref{3}.6. Set $\psi_\infty=\psi\circ\pi_\infty$. Let $(T,M)$ be an M-dynamic in $\texttt{R}_\pi$. There exists an invariant closed subset $N$ of $\texttt{R}_{\psi_\infty}$, which has a dense set of a.p. points such that $M=(\phi\times\phi)N$. Then by Lemma~\ref{3.8}, $(T,N)$ is an M-dynamic in $\texttt{R}_{\psi_\infty}$. Since $\psi_\infty$ is strictly PI, thus $N$ is minimal by using the PI-tower of $\psi_\infty$ and Lemma~\ref{3.9}, and $M$ is also minimal. Then (1) implies (2).

\item $(2)\Rightarrow(3)$: Obvious.

\item $(3)\Rightarrow(1)$: This follows easily from Theorem~\ref{3.5}. The proof is complete.
\end{proof}

\begin{se}[Ellis weak-mixing extension]\label{3.11}
Recall that for $\pi\colon\mathscr{X}\rightarrow\mathscr{Y}$:
\begin{enumerate}[$\bullet$]
\item $\mathscr{X}$ is called an \textit{Ellis $\pi$-weak-mixing extension} of $\mathscr{Y}$ or simply $\pi$ is \textit{Ellis weak-mixing}~\cite{G05, A04} if $\mathbb{F}=\mathbb{A}\mathbb{F}^\prime$, where $\mathbb{A}$ and $\mathbb{F}$ are defined as in Def.~\ref{3}.1.
\end{enumerate}

It is well known that if there is no the Bronstein condition and $Y\not=\{*\}$, an Ellis weak-mixing extension of minimal compact flows need not be a weak-mixing extension (see \cite{G05}); and, an Ellis weak-mixing minimal compact semiflow need not be weakly mixing (see \cite[Exa.~4.3.2]{CD20}).

Moreover, by using the canonical Ellis-Glasner-Shapiro construction of RIC extensions~\cite{EGS, G76, V77, A88, CD20} we can easily obtain the following:

\begin{lem}\label{3.12}
If $\pi$ is Ellis weak-mixing non-proximal, then it is not PI so that there exists non-minimal canonical M-dynamic $(T,\texttt{M}_\infty)$ in $\texttt{R}_\pi$.
\end{lem}

\begin{proof}
Since $\mathbb{A}(\mathbb{A}\mathbb{F}^\prime)^\prime=\mathbb{A}\mathbb{F}^{\prime\prime}$ (by \cite[Lem.~X.4.1]{G76}), so $\mathbb{F}=\mathbb{F}_1\mathbb{A}=\mathbb{F}_2\mathbb{A}=\dotsm=\mathbb{F}_\infty\mathbb{A}$. If $\pi$ were PI, then $\mathbb{F}=\mathbb{A}$ and $\pi$ would be proximal, contrary to the hypothesis $\pi$ non-proximal. Thus, $\pi$ is not PI.
\end{proof}

Notice that the definition of Ellis weak-mixing is independent of the choice of $\pi_X$ and $\pi_Y$. Moreover, a proximal minimal flow is Ellis weak-mixing and PI~(cf.~\cite{G76, CD20}).
\end{se}

\section{Li-Yorke chaos of extension, sensitivity and asymptotically a.p. motions}\label{4}
We will introduce the definition of Li-Yorke chaotic extension in $\S$\ref{4.1}, provide a sufficient condition for sensitivity on initial conditions in $\S$\ref{4.2}, and consider asymptotically a.p. motions in $\S$\ref{4.3}.

\begin{se}[Li-Yorke chaos of extension]\label{4.1}
Let $\pi\colon\mathscr{X}\rightarrow\mathscr{Y}$ be a nontrivial extension of minimal compact metric dynamics. Let $\rho$ be a metric on $X$.

Then $(x,x^\prime)\in\texttt{P}_\pi$ iff $(x,x^\prime)\in \texttt{R}_\pi$ and there is a sequence $\{t_n\}_{n=1}^\infty$ in $T$ with $\lim\limits_{n\to\infty}\rho(t_nx,t_nx^\prime)=0$.
Since $\pi$ is nontrivial, so $\texttt{R}_\pi\not=\Delta_X$. Recall that whenever $T=\mathbb{Z}_+$ and $(x,x^\prime)\in X\times X$ with $x\not=x^\prime$, we say $(x,x^\prime)$ is a Li-Yorke pair iff
$${\liminf}_{n\to\infty}\rho(nx,nx^\prime)=0\quad \textrm{and}\quad {\limsup}_{n\to\infty}\rho(nx,nx^\prime)>0.$$
In view of this, the definition of Li-Yorke chaos has been generalized to group actions; see, e.g., \cite{KL, R, DT, A18} and so on. Moreover, Devaney's chaos implies the Li-Yorke's chaos (cf., e.g.,~\cite{HY} for $\mathbb{Z}$-flow, \cite[Prop.~2.21]{DT} and \cite{A18} for $T$ an abelian group).
We now introduce the Li-Yorke chaos for extensions of minimal dynamics as follows:

\begin{4.1a}
Let $L$ be a closed subset of $X\times X$ such that $\Delta_X\varsubsetneq L\subseteq \texttt{R}_\pi$.
\begin{enumerate}[(1)]
\item We say $(x,x^\prime)\in X\times X$ is a \textit{Li-Yorke pair for $\pi$ rel.} $L$ if $(x,x^\prime)\in \texttt{R}_\pi$ and $L\subseteq\overline{T(x,x^\prime)}$.
\item A set $S$ in $X$ is called \textit{Li-Yorke scrambled for $\pi$ rel. $L$} if every pair $(x,x^\prime)\in S\times S\setminus\Delta_X$ is a Li-Yorke pair for $\pi$ rel. $L$.
\end{enumerate}
\end{4.1a}

\begin{4.1b}
Let $\mathscr{L}=(T,L)$ be an M-dynamic in $\texttt{R}_\pi$, and define $\texttt{Trans}\,(\mathscr{L})$ to be the set
$$
\texttt{Trans}\,(\mathscr{L})=\{\tilde{x}\in L\,|\,\overline{T\tilde{x}}=L\}.
$$
If $\Delta_X\varsubsetneq L$, then each point of $\texttt{Trans}\,(\mathscr{L})$ is a Li-Yorke pair for $\pi$ rel. $L$.
\end{4.1b}

Since $\mathscr{L}$ is an M-dynamic such that $\Delta_X\varsubsetneq L\subseteq \texttt{R}_\pi$, so if $(x,x^\prime)\in\texttt{Trans}\,(\mathscr{L})$, then the set of $t\in T$ at which $Tx$ and $Tx^\prime$ are sufficiently close and the set of $t\in T$ at which $Tx$ and $Tx^\prime$ are far away are both very big (it is at least syndetic). In view of this, the Li-Yorke chaos here is much stronger than the classical sense.
\end{se}

\begin{se}[Sensitive to initial conditions]\label{4.2}
Let $\mathscr{X}$ be a dynamic with $X$ a Hausdorff uniform space not necessarily compact in this part.

\begin{4.2a}
$\mathscr{X}$ is said to be \textit{bounded} if given $\alpha\in\mathscr{U}_X$ there exists a point $x_0\in X$ and a compact set $K$ in $T$ such that $\alpha[Kx_0]=X$. If $\mathscr{X}$ is not bounded, then it called \textit{unbounded}.
\end{4.2a}

Clearly, if $(X,\rho)$ is a metric space such that there is a point $x_0\in X$ such that $\{\rho(x_0,x)\,|\,x\in X\}$ is unbounded in $\mathbb{R}$, then $\mathscr{X}$ is unbounded. Moreover, we can easily construct a compact metric space that can support a unbounded pointwise a.p. flow.
By virtue of these, our ``unbounded'' condition is then much general than the one introduced in \cite[($D_1$)]{Se} for only metric spaces.

In fact,
\begin{(i)}
If $\mathscr{X}$ is bounded having a dense set of a.p. points, then $\mathscr{X}$ is T.T. Moreover, if $\mathscr{X}$ is bounded, then it is totally bounded.
\end{(i)}

\begin{proof}
Suppose this is not true. Then $\mathscr{X}$ would be bounded and thus it would be T.T. Indeed, let $U,V$ be nonempty open sets in $X$. Then there exists a point $x_0\in X$ and $s,t\in T$ such that $sx_0\in U$ and $tx_0\in V$. Further, there exists an open $W\in\mathfrak{N}_{x_0}$ such that $sW\subseteq U$ and $tW\subseteq V$. Since the a.p. points are dense in $X$, we can pick an a.p. point, say $x_1$, in $sW$ and some $s_1\in T$ such that $s_1x_1\in W$. So $ts_1x_1\in V$ and $N_T(U,V)\not=\emptyset$.
\end{proof}

\begin{(ii)}
Let $T$ be $\sigma$-compact. If $\mathscr{X}$ is a pointwise a.p. non-minimal compact dynamic, then $\mathscr{X}$ is unbounded.
\end{(ii)}

\begin{proof}
If this is false, then $\mathscr{X}$ is T.T. by (i) and thus, $\mathscr{X}$ would be minimal by (\ref{s1}.7a) in $\S$\ref{s1}.
\end{proof}

\begin{4.2b}
We say that $\mathscr{X}$ is \textit{sensitive to initial conditions on} $X_0$, where $X_0\subseteq X$, if there exists an index $\varepsilon\in\mathscr{U}_X$ such that for every $x\in X_0$ and all $U\in\mathfrak{N}_x(X)$ there exists some $y\in U$ and some $t\in T$ with $(tx,ty)\notin\varepsilon$. If $\mathscr{X}$ is T.T. and ``periodic points'' are dense in $X$, then we say $\mathscr{X}$ is \textit{Devaney chaotic}.

Then $\mathscr{X}$ is sensitive iff we can find $\varepsilon\in\mathscr{U}_X$ such that for every open set $U\not=\emptyset$ in $X$, there exist $x_1,x_2\in U$ and $t\in T$ with $(tx_1,tx_2)\notin\varepsilon$. So, if $\mathscr{X}$ is sensitive, then $X$ has no isolated point.
\end{4.2b}

\begin{4.2c}
If $X_0$ is dense in $X$ and $\mathscr{X}$ is sensitive to initial conditions on $X_0$, then $\mathscr{X}$ is sensitive to initial conditions on $X$.
\end{4.2c}

\begin{proof}
Obvious.
\end{proof}

\begin{4.2d}[{cf.~\cite[Thm.~1, Thm.~1.1]{Se} for $T=\mathbb{R}_+$ or $\mathbb{Z}_+$ with $X$ a unbounded metric space, under a stronger definition of a.p. points}]
Let $\mathscr{X}$ be a unbounded M-flow; then $\mathscr{X}$ is sensitive to initial conditions on $X$.
\end{4.2d}

\begin{proof}
Let $X_0$ be the set consist of a.p. points of $\mathscr{X}$. Then $X_0$ is dense in $X$. By Lemma~\ref{4.2}c, it is enough to show that $\mathscr{X}$ is sensitive to initial conditions on $X_0$. Suppose this is false. Let $\alpha$, $\varepsilon\in\mathscr{U}_X$ with $\varepsilon^2\subseteq\alpha$. Then there is an index $\delta\in\mathscr{U}_X$ and a point $x_0\in X_0$ such that for all $y\in\delta[x_0]$, $(ty,tx_0)\in\varepsilon$ for all $t\in T$. Moreover, since $x_0$ is a.p. under $\mathscr{X}$, there is a compact set $K$ in $T$ such that
$KN_T(x_0,\delta[x_0])=T$. Thus
$T\delta[x_0]\subseteq\varepsilon[K\delta[x_0]]\subseteq\varepsilon^2[Kx_0]$ and $T\delta[x_0]\subseteq\alpha[Kx_0]$.
Since $\mathscr{X}$ is T.T., $X=\overline{T\delta[x_0]}$ and $X=\overline{\alpha[Kx_0]}$. This is contrary to that $X$ is unbounded. The proof is completed.
\end{proof}

It should be noted that the unboundedness of $\mathscr{X}$ is critical for Theorem~\ref{4.2}d; for example, a minimal equicontinuous compact flow is a nonsensitive M-flow.

\begin{4.2e}[{cf.~\cite[Thm.~1.2]{Se} for $T=\mathbb{Z}_+$ and $\mathbb{R}_+$ and $\texttt{RE}(\mathscr{X})$ dense}]
Let $\mathscr{X}$ be an M-flow with $X$ a complete noncompact space such as $X=\mathbb{R}^n$. Then $\mathscr{X}$ is sensitive to initial conditions on $X$.
\end{4.2e}

\begin{proof}
First note that $\mathscr{X}$ is unbounded; for otherwise, $X$ would be totally bounded complete so that $X$ would be compact. Then by Theorem~\ref{4.2}d, $\mathscr{X}$ is sensitive.
\end{proof}

\begin{4.2f}
A point $x\in X$ is called a \textit{Birkhoff recurrent point} under $\mathscr{X}$\cite{NS, Se, ChD} if given $\varepsilon\in\mathscr{U}_X$ there exists a compact set $K$ in $T$ such that $Tx\subseteq\varepsilon[Ktx]$ for all $t\in T$. The set of Birkhoff recurrent points of $\mathscr{X}$ is denoted $\texttt{RE}(\mathscr{X})$ in \cite{Se}.

A Birkhoff recurrent point must be an a.p. point. In fact, in compact dynamics a point is a.p. iff it is Birkhoff recurrent (cf.~\cite[Thm.~4.1]{ChD}); in locally compact metric flows, a point is a.p. iff it is Birkhoff recurrent (cf.~\cite[Cor.~4.2]{ChD}). However, in locally compact metric semiflows an a.p. point need not be Birkhoff recurrent (see \cite[Rem.~3.15(c)]{AD} for a counterexample).

So in Theorem~\ref{4.2}d, our condition ``a.p. points are dense in $X$'' is generally weaker than Seifert's condition ``$\texttt{RE}(\mathscr{X})$ is dense in $X$'' in \cite{Se}.
\begin{enumerate}[$\bullet$]
\item Let $\mathscr{X}$ be a flow. If $x\in X$ such that given $\varepsilon\in\mathscr{U}_X$ there exists a compact set $L\subseteq T$ with $\varepsilon[tx]\cap Latx\not=\emptyset$ for all $t,a\in T$, then $x\in \texttt{RE}(\mathscr{X})$ (cf.~\cite[p.~1721]{Se} for $T=\mathbb{R}_+$ and $\mathbb{Z}_+$).
\end{enumerate}
\begin{proof}
Let $s\in T$ and $t\in T$. Select $a\in T$ with $s=at$ or equivalently $t=a^{-1}s$. Then $tx\in\varepsilon[Lsx]$. This shows that $x\in \texttt{RE}(\mathscr{X})$.
\end{proof}

Thus, $x$ is Birkhoff recurrent under $\mathscr{X}$ iff $(T,\overline{Tx})$ is a weakly a.p. subflow of $\mathscr{X}$ (in the sense of Gottschalk \cite{AD}: Given $\varepsilon\in\mathscr{U}_X$ there exists a compact set $L$ in $T$ such that for all $y\in\overline{Tx}$, $Lt\cap N_T(y,\varepsilon[y])\not=\emptyset\ \forall t\in T$).
\end{4.2f}

\begin{4.2g}
By $\texttt{Equi}(\mathscr{X})$ we denote the set of points at which $\mathscr{X}$ is equicontinuous/$T$-stable. That is, $x\in\texttt{Equi}(\mathscr{X})$ iff given $\varepsilon\in\mathscr{U}_X$ there is a member $U\in\mathfrak{N}_x$ such that $(tx,ty)\in\varepsilon$ for all $t\in T$ and all $y\in U$. See Def.~IV and Def.~a in $\S$\ref{s2.2.1}.
\end{4.2g}

\begin{4.2h}
Let $\mathscr{X}$ be a unbounded T.T. flow. If $\texttt{Equi}\,(\mathscr{X})\not=\emptyset$, then the a.p. points are not dense in $X$ (and consequently $\texttt{RE}(\mathscr{X})$ is not dense in $X$ and $\mathscr{X}$ is not an M-flow).
\end{4.2h}

\begin{proof}
Suppose this is false. Then by Theorem~\ref{4.2}d, $\mathscr{X}$ would be sensitive to initial conditions on $X$. So there would be an index $\varepsilon\in\mathscr{U}_X$ such that for $x\in\texttt{Equi}(\mathscr{X})$ and $U\in\mathfrak{N}_x$ there exists a point $y\in U$ with $(tx,ty)\notin \varepsilon$ for some $t\in T$. This contradicts equicontinuity of $\mathscr{X}$ at $x$. The proof is completed.
\end{proof}

\begin{4.2i}
If ``a.p.'' points are replaced by ``Birkhoff recurrent'' points, then Theorem~\ref{4.2}d and Corollary~\ref{4.2}h still hold for $T$ a semigroup. By virtue of this, Corollary~\ref{4.2}h is in conflict with an open question in \cite[p.~1725]{Se}.
\end{4.2i}

Using ``$\mathscr{X}$ non-minimal'' instead of ``$\mathscr{X}$ unbounded'', by an argument different with that of Theorem~\ref{4.2}d and \cite{HY,CC,SK}, we can conclude the following:

\begin{4.2j}[{cf.~\cite{GW} for compact M-dynamics; \cite{B92,HY,CC,SK} for Devaney chaotic dynamics}]
Let $\mathscr{X}$ be a non-minimal M-dynamic. Then $\mathscr{X}$ is sensitive to initial conditions on $X$.
\end{4.2j}

\begin{proof}
Since $X$ is a regular space and $\mathscr{X}$ is non-minimal, there exist $\varepsilon\in\mathscr{U}_X$, $y_0\in X$, and $M$ a minimal set in $X$ such that $(y_1,y_2)\notin\varepsilon$ for all $y_1\in\varepsilon[M]$ and all $y_2\in\varepsilon[y_0]$. Pick $\alpha\in\mathscr{U}_X$ with $\alpha^2\subseteq\varepsilon$. To show $\mathscr{X}$ sensitive, suppose this is false. Then we can choose $x_0\in X$ and $\delta\in\mathscr{U}_X$ such that $y_1,y_2\in\delta[x_0]$ implies that $(ty_1,ty_2)\in\alpha$ for all $t\in T$. Noting $N_T(\delta[x_0],\varepsilon[M])$ is syndetic thick in $T$ and $N_T(\delta[x_0],\varepsilon[y_0])$ is syndetic in $T$, we can pick $s\in N_T(\delta[x_0],\varepsilon[M])\cap N_T(\delta[x_0],\varepsilon[y_0])$. Moreover, we can pick $y_1,y_2\in\delta[x_0]$ such that $sy_1\in\varepsilon[M]$ and $sy_2\in\varepsilon[y_0]$, and $(sy_1,sy_2)\in\alpha$, contrary to $(sy_1,sy_2)\notin\varepsilon$ for $\alpha^2\subseteq\varepsilon$. The proof is completed.
\end{proof}

\begin{4.2k}
Since $X$ need not be locally compact, we don't know whether or not a unbounded M-flow is non-minimal. On the other hand, any nonequicontinuous minimal flow is sensitive. So ``non-minimal'' is not a necessary condition for sensitivity, and Theorem~\ref{4.2}d is not a corollary of Theorem~\ref{4.2}j.
\end{4.2k}
\end{se}

\begin{se}[Asymptotically a.p. motions]\label{4.3}
In this part, let $\mathscr{X}$ be a semiflow with $X$ a Hausdorff uniform space. We first define
\begin{enumerate}[(i)]
\item the $\omega$-limit set $\omega[x]=\{w\in X\,|\,\forall\, U\in\mathfrak{N}_w\textrm{ and }s\in T, \exists\,t\in T\textrm{ s.t. }tsx\in U\}$. Clearly, for $T=\mathbb{R}_+$ or $\mathbb{Z}_+$, $w\in\omega[x]$ iff $\exists\, t_n\to\infty$ such that $t_nx\to w$. Moreover, the $P$-limit set of $x$, $P_{\!x}=\omega[x]$, for $P=T$. If $T$ is a group, then $\omega[x]=\overline{Tx}$ will give us no new information.

\item $\texttt{AAP}(\mathscr{X})=\{x\in X\,|\, \exists\textrm{ an a.p. point }y\textrm{ s.t. }\forall\,\varepsilon\in\mathscr{U}_X, \exists\, s\in T\textrm{ s.t. }(tsx,tsy)\in\varepsilon\,\forall t\in T\}$. Clearly, for $T=\mathbb{R}_+$ or $\mathbb{Z}_+$ and $(X,\rho)$ a metric space, $x\in\texttt{AAP}(\mathscr{X})$ iff $\exists$ an a.p. point $y$ such that $\rho(tx,ty)\to 0$ as $t\to\infty$.
\end{enumerate}
We call $\texttt{AAP}(\mathscr{X})$ the \textit{set of asymptotically a.p. points of} $\mathscr{X}$. By virtue of \ref{4.2}f, this is equivalent to that defined by Seifert~\cite[(iv)]{S1} under compact ambit.

\begin{4.3a}
Let $T$ be such that $\bigcap_{s\in F}Ts\not=\emptyset$ for every finite set $F\subset T$. If $x\in X$ such that $\overline{Tx}$ is compact, then $\omega[x]\not=\emptyset$ is compact invariant.
\end{4.3a}

\begin{proof}
By $\omega[x]=\bigcap_{s\in T}\overline{Tsx}$ and compactness, it follows easily that $\omega[x]\not=\emptyset$ is compact. The invariance is evident.
\end{proof}

\begin{4.3b}
Let $T$ be such that $\bigcap_{s\in F}Ts\not=\emptyset$ for every finite set $F\subset T$. Let $x\in X$ such that $\overline{Tx}$ is compact and $\omega[x]\subseteq\texttt{Equi}\,(\mathscr{X})$. Then $(T,\omega[x])$ is a minimal subsemiflow of $\mathscr{X}$, $\omega[x]\subseteq \texttt{RE}(\mathscr{X})$, and $x\in\texttt{AAP}(\mathscr{X})$. If, in addition, $T$ is an amenable semigroup, then $(T,\omega[x])$ is an a.p. subdynamic of $\mathscr{X}$.
\end{4.3b}

\begin{proof}
By Lemma~\ref{4.3}a, $\omega[x]$ is an invariant compact nonempty subset of $X$. Let $M$ be a minimal subset of $\omega[x]$. If $M\not=\omega[x]$, then we can take a point $w\in\omega[x]\setminus M$. Further, there exists an index $\varepsilon\in\mathscr{U}_X$ such that $\varepsilon[w]\cap\varepsilon[M]=\emptyset$. Let $y\in M$. Select $U\in\mathfrak{N}_y$ with $U\subset\varepsilon[M]$ such that if $sx\in U$ then $tsx\in\varepsilon[M]$ for all $t\in T$. This is contrary to $w\in\omega[x]$. Thus $\omega[x]$ is a minimal set in $X$. This also shows that $\omega[x]\subseteq \texttt{RE}(\mathscr{X})$.

Now we will prove that $x\in\texttt{AAP}(\mathscr{X})$. By a theorem of Auslander-Ellis (cf., e.g., \cite[Thm.~8.7]{F81}), we can find a point $y\in\omega[x]$ such that $x$ is proximal to $y$. Let $\varepsilon\in\mathscr{U}_X$. We note here that since $\omega[x]\subseteq\texttt{Equi}(\mathscr{X})$ is compact, there is an index $\delta\in\mathscr{U}_X$ such that $(z,w)\in\delta$ and $w\in\omega[x]$ implies that $(tz,tw)\in\varepsilon$ for all $t\in T$. Now, we can take some $s\in T$ such that $(sx,sy)\in\delta$, and further, $(tsx,tsy)\in\varepsilon$ for all $t\in T$. Thus $x\in\texttt{AAP}(\mathscr{X})$.

Finally, If $T$ is an amenable semigroup, then by \cite[Thm.~1.15]{AD} it follows that $(T,\omega[x])$ is an a.p. subdynamic of $\mathscr{X}$.
The proof is completed.
\end{proof}

\begin{4.3c}[{cf.~\cite[Thm.~1]{S1} for $d=1$}]
Let $T=\mathbb{R}_+^d$ or $\mathbb{Z}_+^d$ where $d\ge1$; let $\overline{Tx}$ be compact such that $\omega[x]\subseteq\texttt{Equi}\,(\mathscr{X})$. Then $x\in\texttt{AAP}(\mathscr{X})$, $\omega[x]\subset \texttt{RE}(\mathscr{X})$, and $(T,\omega[x])$ is an a.p. subsemiflow of $\mathscr{X}$.
\end{4.3c}

\begin{proof}
Clearly, $\mathbb{R}_+^d$ and $\mathbb{Z}_+^d$ both satisfy the condition of Lemma~\ref{4.3}a (by Remark~\ref{4.3}g below). Then Theorem~\ref{4.3}b follows Corollary~\ref{4.3}c.
\end{proof}

\begin{4.3d}
Let $x\in X$ such that
\begin{enumerate}[(1)]
\item $\omega[x]\not=\emptyset$ is compact with $\omega[y]=\omega[x]$ for all $y\in\omega[x]$ and
\item $\omega[x]\cap\texttt{Equi}\,(\mathscr{X})\not=\emptyset$.
\end{enumerate}
Then $x\in\texttt{AAP}(\mathscr{X})$.
\end{4.3d}

\begin{proof}
By condition (1), $\omega[x]$ is a compact minimal set of $\mathscr{X}$. Then from condition (2), it follows that $\omega[x]\subseteq\texttt{Equi}\,(\mathscr{X})$. Now using Theorem~\ref{4.3}b, we have that $x\in\texttt{AAP}(\mathscr{X})$.
\end{proof}

\begin{4.3e}[{cf.~\cite[Lem.~1]{S1} for $T=\mathbb{R}_+$}]
In Theorem~\ref{4.3}b, if $T$ is an amenable semigroup, then for $x\in\texttt{AAP}(\mathscr{X})$, the corresponding shadowing a.p. motion as in Def.~\ref{4.3}(ii) is unique.
\end{4.3e}

\begin{proof}
Suppose this is false. Then there are two distinct points $y_1,y_2\in\omega[x]$ such that for every $\varepsilon\in\mathscr{U}_X$ there are $s_1,s_2\in T$ such that $(ts_ix,ts_iy_i)\in\varepsilon\ \forall t\in T$, for $i=1,2$. By $Ts_1\cap Ts_2\not=\emptyset$ for all $s_1,s_2\in T$, it follows that $y_1$ is proximal to $y_2$ under $\mathscr{X}$. However, since $y_1,y_2\in\omega[x]$  and $(T,\omega[x])$ is a.p., thus $(y_1,y_2)$ is a.p. under $\mathscr{X}\times \mathscr{X}$ so that $y_1=y_2$. The proof is complete.
\end{proof}

It should be mentioned that our unicity proof here is much simpler than Seifert's proof for the special case that $T=\mathbb{R}_+$ and $X$ a complete metric space \cite[Lem.~1]{S1}. There Seifert need to use $y_1,y_2\in \texttt{RE}(\mathscr{X})$ instead of our ``proximality'' argument.

\begin{4.3f}
Recall that a topological semigroup $T$ is called a ``$C$-semigroup'' \cite{KM} if $\overline{T\setminus Ts}$ is compact for all $s\in T$. For example, $(\mathbb{R}_+,+)$ and $(\mathbb{Z}_+,+)$ are $C$-semigroups. Then:
\begin{enumerate}[$\bullet$]
\item If $T$ is a noncompact $C$-semigroup, then $\bigcap_{i=1}^nTs_i\not=\emptyset$ for all $n\ge2$ and $s_1,\dotsc,s_n\in T$.
\end{enumerate}
\end{4.3f}

\begin{proof}
Let $n\ge2$ and $s_1,\dotsc,s_n\in T$. If $\bigcap_{i=1}^nTs_i=\emptyset$, then $\left(T\setminus Ts_1\right)\cup\dotsm\cup\left(T\setminus Ts_n\right)=T$ so that $T$ is compact, contrary to the hypothesis. Thus $\bigcap_{i=1}^nTs_i\not=\emptyset$.
\end{proof}

\begin{4.3g}
Let $T=\mathbb{R}_+^d$ and $\mathbb{Z}_+^d$, where $d\ge1$ is an integer. Then $\bigcap_{i=1}^n(T+s_i)\not=\emptyset$ for all $n\ge2$ and $s_1,\dotsc,s_n\in T$.
\end{4.3g}

\begin{proof}
Given $n\ge2$, let $K=\{o, s_1,\dotsc,s_n\}$, which is a compact subset of $G$, where $G=\mathbb{R}^d$ or $\mathbb{Z}^d$ and $o=(0,\dotsc,0)$. Noting that $T$ is a thick semigroup in $G$. Then by Lemma~\ref{2.7}a, $\bigcap_{i=1}^n(T+s_i)$ is a thick set in $G$ so that $\bigcap_{i=1}^n(T+s_i)\not=\emptyset$. The proof is completed.
\end{proof}
\end{se}

\section{Li-Yorke chaos of non-PI extensions}\label{5}

If $\pi\colon\mathscr{X}\rightarrow\mathscr{Y}$ is a non-PI extension of minimal compact flows, then we can define a canonical M-flow $(T,\texttt{M}_\infty)$ such that $\Delta_X\varsubsetneq \texttt{M}_\infty\subseteq \texttt{R}_\pi$ as in Def.~\ref{3.4} in $\S$\ref{3}. Moreover, $\texttt{M}_\infty$ has been described by Theorem~\ref{3.7}.
We can then state our main result of this section as follows:

\begin{thm}\label{5.1}
Let $\pi\colon\mathscr{X}\rightarrow\mathscr{Y}$ be a non-PI extension of minimal compact metric flows. Then there exists a residual set $Y_{\textrm{LY}}$ in $Y$ such that for every $y\in Y_{\textrm{LY}}$ and for generic $x\in\pi^{-1}y$ there is a residual set $S[x]$ in $\texttt{M}_\infty[x]$ with the property that $(x,x^\prime)$ is a Li-Yorke pair for $\pi$ rel. $\texttt{M}_\infty$ for each $x^\prime\in S[x]$.
\end{thm}

Recall that a mapping is said to be \textit{semi-open} if the image under this mapping of every open nonempty set contains an open nonempty set.
If $\phi\colon\mathscr{X}\rightarrow\mathscr{Y}$ is an extension of compact flows with $Y$ minimal and $X$ having a dense set of a.p. points, then $\phi$ is semi-open (\cite[Lem.~3.12.15]{B79}, \cite[Thm.~1.15]{A88}).

To prove Theorem~\ref{5.1}, we shall need two auxiliary lemmas. The first is a special case of the mentioned semi-openness theorem.

\begin{lem}\label{5.2}
Let $\pi\colon\mathscr{X}\rightarrow\mathscr{Y}$ be an extension of compact flows with $Y$ minimal. If $(T,L)$ is an M-flow in $\texttt{R}_\pi$, then $p\colon L\rightarrow Y$ defined by $(x,x^\prime)\mapsto \pi x$ is semi-open onto.
\end{lem}

\begin{lem}[{a topological `Fubini theorem'; cf.~\cite[Prop.~3.1]{V70}, \cite[Lem.~5.2]{G90} for other versions}]\label{5.3}
Let $\phi\colon W\rightarrow Z$ be a semi-open continuous onto map of Polish spaces. Suppose $K$ is a residual subset of $W$. Let
$Z_K=\{z\in Z\,|\,K\cap\phi^{-1}z\textrm{ is a residual subset of }\phi^{-1}z\}$.
Then $Z_K$ is residual in $Z$.
\end{lem}

\begin{proof}
Since $K$ is residual in $W$, there exists a sequence $F_1, F_2,\dotsc$ of closed nowhere dense sets in $W$ such that $W\setminus K\subseteq\bigcup_{i=1}^\infty F_i$. Let $K_z=K\cap\phi^{-1}z$ for all $z\in Z$, and we notice that
$$
\phi^{-1}z\setminus K_z\subseteq{\bigcup}_{i=1}^\infty(F_i\cap\phi^{-1}z).
$$
Now let
$$
B=\{z\in Z\,|\,\exists\,i\ge1\textrm{ s.t. }\textrm{int}_{\phi^{-1}z}(F_i\cap\phi^{-1}z)\not=\emptyset\}.
$$
So if $z\notin B$, then $K_z$ is of course residual in $\phi^{-1}z$. Then $Z\setminus Z_K\subseteq B$. It will therefore suffice to prove $B$ is of the first category.

Let $U_1, U_2, \dotsc$ be a basis for the topology on $W$. If $z\in B$, then we have for some integers $m$ and $i$ that $\emptyset\not=U_m\cap\phi^{-1}z\subseteq F_i$. Put
$$
C_{mi}=\{z\in B\,|\,\emptyset\not=U_m\cap\phi^{-1}z\subseteq F_i\}\quad \textrm{for }1\le m,i<\infty.
$$
Then
$B=\bigcup_{m,i=1}^\infty C_{mi}$
and $B$ is a first category set in $Z$ if each set $C_{mi}$ is nowhere dense in $Z$.

Let $D_{mi}=\textrm{int}_Z\textrm{cls}_ZC_{mi}$, and suppose $D_{mi}\not=\emptyset$ for some two positive integers $m,i$. Since $C_{mi}\cap D_{mi}\not=\emptyset$, and, $U_m\cap\phi^{-1}z\not=\emptyset$ whenever $z\in C_{mi}\cap D_{mi}$, hence the set
$$
F:=U_m\cap \phi^{-1}D_{mi}
$$
is an open nonempty subset of $U_m$. We shall prove $F\subseteq F_i$, contrary to that $F_i$ is nowhere dense, then Lemma~\ref{5.3} will follow.

Let $W_o=\{w\in W\,|\,\phi\textrm{ is open at }w\}$. Since $\phi$ is semi-open and $W$ is a Polish space, $W_o$ is a residual subset of $W$. Indeed, let
$G_n=\{x\,|\,x\in W\textrm{ s.t. }x\in B_{1/n}(x)\cap\phi^{-1}\textrm{int}_W\phi B_{1/n}(x)\}$ for $n=1,2,\dotsc$; then $G_n$ is open dense in $W$, and, $W_o=\bigcap_nG_n$.

Now put $F_o=F\cap W_o$. Then $F_o$ is dense in $F$. If $w\in F_o$ and $z=\phi w\in C_{mi}$, then by definition $w\in U_m\cap\phi^{-1}z\subseteq F_i$. Now let $w\in F_o$ but $z=\phi w\notin C_{mi}$. Then at least $z\in\textrm{cls}_ZC_{mi}$, or, $z=\lim_{n\to\infty}z_n$, $z_n\in C_{mi}$. So there exists a sequence $w_n\to w$ in $F$ such that $\phi w_n=z_n$ for all $n$. Eventually, $w_n\in U_m$, and since $z_n\in C_{mi}$, we have $w_n\in F_i$ for all $n$. Therefore $w\in \overline{F}_i=F_i$ and $F_o\subseteq F_i$. Since $F_i$ is closed and $F_o$ is dense in $F$, so $F\subseteq F_i$. The proof is complete.
\end{proof}

Notice here that comparing with Veech~\cite[Prop.~3.1]{V70} and Glasner~\cite[Lem.~5.2]{G90}, the point of Lemma~\ref{5.3} is that $W$ need not be a \textbf{minimal} flow. Moreover, as a result of Lemma~\ref{5.3}, we can conclude the following useful result:

\begin{cor*}[{\cite[Lem.~3.2]{HY}}]
Let $E$ be a Polish space and $R$ a relation on $E$ such that $R$ contains a residual subset of $E\times E$. Then there exists a residual subset $A$ of $E$ such that for all $x\in A$ there exists a residual subset $K_x$ of $E$ with $\{x\}\times K_x\subset R$.
\end{cor*}

\begin{proof}[\textbf{Proof of Theorem~\ref{5.1}}]
Since $\mathscr{M}_\infty:=(T,\texttt{M}_\infty)$ is T.T. by Theorem~\ref{3.5} and $\texttt{M}_\infty$ is a compact metric space, $\texttt{Trans}(\mathscr{M}_\infty)$ is residual in $\texttt{M}_\infty$. By Lemma~\ref{4.1}b, each point of $\texttt{Trans}(\mathscr{M}_\infty)$ is a Li-Yorke pair for $\pi$ rel. $\texttt{M}_\infty$.
Let $\lambda\colon\mathscr{M}_\infty\rightarrow\mathscr{X}$ be defined by $(x,x^\prime)\mapsto x$. Clearly $\lambda$ is an extension of flows.
Since $\lambda\colon \texttt{M}_\infty\rightarrow X$ is semi-open onto by Lemma~\ref{5.2}, there exists by Lemma~\ref{5.3} a residual set $X_{\textrm{LY}}$ in $X$ such that for every $x\in X_{\textrm{LY}}$, the set $S[x]=\{x^\prime\in X\,|\,(x,x^\prime)\in \texttt{Trans}(\mathscr{M}_\infty)\}$ is relatively residual in $\lambda^{-1}x=\texttt{M}_\infty[x]$.
Furthermore, since $\pi\colon X\rightarrow Y$ is semi-open onto, by Lemma~\ref{5.3} there exists a residual set $Y_{\textrm{LY}}$ in $Y$ such that for every $y\in Y_{\textrm{LY}}$, the set $\pi^{-1}y\cap X_{\textrm{LY}}$ is relatively residual in $\pi^{-1}y$. The proof is complete.
\end{proof}

\begin{lem}[{cf.~\cite[Lem.~3.1]{HY} for $E$ having no isolated point $\&$ $R\subset E\times E$}]\label{5.4}
Let $E$ be a Polish space.
Let $R\subset E\times E\setminus\Delta_E$ be a symmetric relation on $E$ with the property that there exists a residual subset $A$ of $E$ such that for each $a\in A$, $R[a]$ contains a residual subset of $E$. Then there is a dense uncountable subset $B$ of $E$ such that $B\subseteq A$ and $B\times B\setminus\Delta_E\subseteq R$.
\end{lem}

\begin{proof}
Let $\mathfrak{B}$ be the family of dense subsets $B$ of $E$ such that $B\times B\setminus\Delta_E\subseteq R$ and $B\subseteq A$. Clearly, $\mathfrak{B}\not=\emptyset$. Indeed, if $B\subset A$ is at most countable with $B\times B\setminus\Delta\subseteq R$ and $\overline{B}\not=E$, then there exists an element $x\in A\cap(E\setminus\overline{B})\cap\left(\bigcap_{b\in B}R[b]\right)$. So $B_1=B\cup\{x\}\subset A$ such that $B_1\times B_1\setminus\Delta_E\subseteq R$. Then by induction we can find a subset of $A$ that belongs to $\mathfrak{B}$.

Now by Zorn's lemma there exists a maximal element $B$ in $(\mathfrak{B},\subseteq)$. Then $B$ is uncountable. Otherwise, we can choose an element
$x\in A\cap\left(\bigcap_{b\in B}R[b]\right)$ such that $B\cup\{x\}\in\mathfrak{B}$, contrary to that $B$ is a maximal element in $\mathfrak{B}$. This proves Lemma~\ref{5.4}.
\end{proof}

Theorem~\ref{5.1} is an extension of \cite[Prop.~3.21]{DT}. As a matter of fact, using $\texttt{R}_\pi$ instead of $\texttt{M}_\infty$ in the above arguments, we can obtain the following:

\begin{thm}\label{5.5}
Let $\pi\colon\mathscr{X}\rightarrow\mathscr{Y}$ be a nontrivial extension of metric flows with $Y$ minimal. Suppose that $\texttt{R}_\pi$ has a dense set of a.p. points and that $\pi$ is weakly mixing. Then there exists a residual set $Y_{\textrm{LY}}$ in $Y$ such that for every $y\in Y_{\textrm{LY}}$, $\pi^{-1}y$ has no isolated points and it contains a relatively dense uncountable Li-Yorke scrambled set for $\pi$ rel. $\texttt{R}_\pi$.
\end{thm}

\begin{proof}
Notice that $\mathscr{R}_\pi$ is an M-flow with $\texttt{R}_\pi\not=\Delta_X$. Let $S=\texttt{Trans}(\mathscr{R}_\pi)$. Then $S$ is a symmetric relation on $X$ such that $S\cap\Delta_X=\emptyset$ and it is residual in $\texttt{R}_\pi$.
Using Lemma~\ref{5.3} for the mappings $\texttt{R}_\pi\xrightarrow{(x,x^\prime)\mapsto\pi x} Y$ and $\pi^{-1}y\times\pi^{-1}y\xrightarrow{(x,x^\prime)\mapsto x}\pi^{-1}y$, it follows that there exists a residual set $Y_{\textrm{LY}}$ in $Y$ such that for all $y\in Y_{\textrm{LY}}$, $\pi^{-1}y$ contains a relatively residual set $(\pi^{-1}y)_\textrm{LY}$ with the property that for $x\in(\pi^{-1}y)_\textrm{LY}$ the cell $S[x]$ is relatively residual in $\pi^{-1}y$.

Given $y\in Y_\textrm{LY}$, put $E=\pi^{-1}y$ and $R=S\cap E\times E$. Then $R$ is a symmetric relation on $E$ with $R\cap\Delta_E=\emptyset$. If $E$ has an (relative) isolated point $x_0$, then $x_0\in(\pi^{-1}y)_\textrm{LY}$ and $x_0\in S[x_0]$ so that $(x_0,x_0)$ is a transitive point for $\mathscr{R}_\pi$. This is impossible for $\texttt{R}_\pi\not=\Delta_X$. This shows that $E$ has no isolated points. Let $A=(\pi^{-1}y)_\textrm{LY}$. Then by Lemma~\ref{5.4}, there exists a dense uncountable subset $B$ of $E$ such that $(x,x^\prime)\in R$ whenever $x,x^\prime\in B$ with $x\not=x^\prime$.

Clearly, we have for all $y\in Y_\textrm{LY}$ that $B\subseteq\pi^{-1}y$ is a Li-Yorke scrambled set for $\pi$ rel. $\texttt{R}_\pi$. The proof is complete.
\end{proof}

Therefore if $\pi\colon\mathscr{X}\rightarrow\mathscr{Y}$ is a nontrivial Bronstein extension of minimal metric flows and if $\pi$ has no nontrivial equicontinuous factors, then $\pi$ is Li-Yorke chaotic in the sense of Theorem~\ref{5.5} (i.e., there is a residual set $Y_{\textrm{LY}}$ in $Y$).

By Theorem~3.7 and Theorem~\ref{5.5}, it follows that if $Y_\textrm{LY}$ is defined by Theorem~\ref{5.1}, then there exists, for all $y\in Y_{\textrm{LY}}$, a uncountable Li-Yorke scrambled set in $\pi^{-1}y$ for $\pi$ rel. $\texttt{M}_\infty$.
When $Y$ is a singleton set, the Bronstein condition on $\pi$ in Theorem~\ref{5}.5 may be removed as follows:

\begin{prop}[{cf.~\cite[Prop.~3.11]{DT}}]\label{5.6}
Let $\mathscr{X}$ be a weak-mixing, nontrivial, metric dynamic. Then $\mathscr{X}$ has a dense uncountable Li-Yorke scrambled set (for $\pi\colon X\rightarrow\{*\}$) rel. $X\times X$.
\end{prop}

\begin{proof}
Since $X\times X$ is a compact metric space, $R:=\texttt{Trans}(\mathscr{X}\times \mathscr{X})$ is residual in $X\times X$ so that $R\cap\Delta_X=\emptyset$. Moreover, as $X\times X\rightarrow X$ defined by $(x,x^\prime)\mapsto x$ is an open onto map, it follows by Lemmas~\ref{5.3} and \ref{5.4} that there exists a dense uncountable set $B$ in $X$ such that if $x,x^\prime\in B$ and $x\not=x^\prime$, then $\overline{T(x,x^\prime)}=X\times X$. Since $X\not=\{*\}$, $\Delta_X\not=X\times X$ and $B$ is a Li-Yorke scrambled set for $\pi\colon X\rightarrow\{*\}$ rel. $X\times X$. The proof is complete.
\end{proof}

\begin{prop}\label{5.7}
Let $\pi\colon\mathscr{X}\rightarrow\mathscr{Y}$ be a nontrivial extension of minimal compact metric flows. Suppose $\pi$ is Ellis weakly mixing and non-proximal. Then there exists a residual set $Y_{\textrm{LY}}$ in $Y$ such that for every $y\in Y_{\textrm{LY}}$, $\pi^{-1}y$ contains a uncountable Li-Yorke scrambled set for $\pi$ rel. $\texttt{M}_\infty$.
\end{prop}

\begin{proof}
This follows easily from Theorem~\ref{5.5} and Lemma~\ref{3.12}.
\end{proof}

Notice that if $\mathscr{X}$ is an Ellis weak-mixing minimal flow, then it is weakly mixing (cf.~\cite{G05,A04}). However, $\texttt{M}_\infty\varsubsetneq X\times X$ in general even if $\mathscr{X}$ is weakly mixing.

\section{M-flows need not be point-transitive and examples}\label{6}

If $\mathscr{X}$ is an M-dynamic, then it is syndetically transitive; that is, $N_T(U,V)$ is a syndetic set in $T$ for all nonempty open sets $U,V$ in $X$. It is well known that if $X$ is not separable, then a T.T. flow need not be point-transitive (cf.~\cite[Exa.~4.17]{EE}). As a matter of fact, an M-flow need not be point-transitive too.

\begin{exa}\label{6.1}
Let $X=Y^T$ be the space of all functions $f\colon T\rightarrow Y$, continuous or not, equipped with the pointwise convergence topology, where $Y$ is a compact Hausdorff space and $T$ is an infinite discrete group with identity $e$. Then $X$ is a compact Hausdorff space.
Given $t\in T$ and $f\in X$, define
$f^t\colon T\rightarrow Y$ by $\tau\mapsto f(\tau t)$.
We now define the right-translation flow $\mathscr{X}$ on $X$ with the phase group $T$ as follows:
$T\times X\rightarrow X$, $(t,f)\mapsto tf:=f^t$.
Then:
\begin{enumerate}[(1)]
\item {\it $\mathscr{X}$ is syndetically and thickly transitive.}
\end{enumerate}

\begin{proof}
For all nonempty open subsets $U,V$ of $Y$ and $\tau_1,\dotsc,\tau_n, s_1,\dotsc,s_n$ in $T$, let
\begin{gather*}
\mathcal{U}=\left[\{\tau_1,\dotsc,\tau_n\},U\right]=\left\{f\in X\,|\,f(\tau_i)\in U, i=1,\dotsc,n\right\}\\ \mathcal{V}=\left[\{s_1,\dotsc,s_n\},V\right].
\end{gather*}
Let
\begin{equation*}
T_0=\left\{s_i^{-1}\tau_j\,|\,i=1,\dotsc,n;\ j=1,\dotsc,n\}\cup\{s_i^{-1}s_j\,|\,i=1,\dotsc,n;\  j=1,\dotsc,n\right\}
\end{equation*}
and
\begin{equation*}
T_1=T\setminus T_0.
\end{equation*}
Since $T$ is an infinite discrete group and $T_0$ is finite, it is easy to check that $T_1$ is syndetic and thick in $T$.
Indeed, if $K=\{e\}\cup t_0T_0^{-1}$ for some $t_0\in T_1$, then $Kt\cap T_1\not=\emptyset$ for all $t\in T=T_0\cup T_1$ so $T_1$ is syndetic in $T$.
If $F$ is any finite subset of $T$, then there exists some $t\in T$ such that $Ft\subseteq T_1$; otherwise, there exists a $k\in F$ and $t\not=s$ in $T$ with $kt, ks\in T_0$ and $kt=ks$ so that $t=s$.

Thus for all $\bt\in T_1$, we have $\{s_1\bt,\dotsc,s_n\bt\}\cap\{\tau_1,\dotsc,\tau_n, s_1,\dotsc,s_n\}=\emptyset$.
Now choose $f\in X$ such that $f(\tau_i)\in U$ and $f(s_i\bt)\in V$ for $1\le i\le n$. Thus $f\in\mathcal{U}$ and $f^{\bt}\in\mathcal{V}$ so that $N_T(\mathcal{U},\mathcal{V})\supseteq T_1$. This proves our assertion.
\end{proof}
\noindent
Since every syndetic subset of $T$ intersects non-voidly every thick subset of $T$, thus we have in fact concluded the following fact:
\begin{enumerate}[(2)]
\item {\it $\mathscr{X}$ is weakly mixing.}
\end{enumerate}

\begin{enumerate}[(3)]
\item {\it Let $Y$ be a non-separable space and $T$ a countable discrete group. Then $\mathscr{X}$ is not point-transitive, namely, $\texttt{Trans}\,(\mathscr{X})=\emptyset$; and, $\texttt{Equi}\,(\mathscr{X})=\emptyset$.}
\end{enumerate}

\begin{proof}
First $X$ has no countable dense subset.
Because $X$ is not separable and $T$ is countable, $\mathscr{X}$ is not point-transitive. Since $\mathscr{X}$ is T.T., $\texttt{Equi}(\mathscr{X})\subseteq\texttt{Trans}(\mathscr{X})=\emptyset$.
\end{proof}

\begin{enumerate}[(4)]
\item {\it Let $T=\mathbb{Z}$ or $\mathbb{R}$ with $e=0$, then $\mathscr{X}$ is an M-flow.}
\end{enumerate}

\begin{proof}
Let $f\in X$ and $k>1$. Define a periodic function $\tilde{f}\colon T\rightarrow Y$ with period $2k+1$ such that $\tilde{f}(t)=f(t)$ for $-k\le t\le k$. So $X$ has a dense periodic points set and $\mathscr{X}$ is an M-flow by (1).
\end{proof}
\noindent
Moreover, as a matter of fact we have concluded the following fact:
\begin{enumerate}[(5)]
\item {\it Let $T=\mathbb{Z}$ or $\mathbb{R}$ with $e=0$, then $\mathscr{X}\times \mathscr{X}$ is an M-flow.}
\end{enumerate}
Here $\mathbb{R}$ is thought of as a discrete group.
This completes the construction of Example~\ref{6.1}.
\end{exa}

We notice that even for $X$ is a compact metric space,
a point-transitive semiflow need not be T.T. in general.
Let us consider a simple counterexample as follows.

\begin{exa}
Let $T=\mathbb{Q}_+$ be the additive semigroup of nonnegative rational numbers, and, let $X=\mathbb{R}_+\cup\{+\infty\}$ be equipped with the usual one-point compactification topology. Define a compact semiflow $\mathscr{X}$ as follows:
$T\times X\rightarrow X$, $(t,x)\mapsto t+x$.
Clearly, $\texttt{Trans}\,(\mathscr{X})=\{0\}$ so $\mathscr{X}$ is point-transitive. But $\mathscr{X}$ is not T.T., for $\texttt{Trans}\,(\mathscr{X})$ is not a dense $G_\delta$-subset of $X$. In fact, for open subsets $U=(10,12)$ and $V=(5,6)$ of $X$, we have $N_T(U,V)=\emptyset$.
\end{exa}

\begin{exa}\label{6.3}
Let $\mathbbm{k}\ge2$ be an integer and $\mathbb{Z}_\mathbbm{k}=\mathbb{Z}/\mathbbm{k}\mathbb{Z}$ endowed with the discrete topology. Let $\Gamma$ be an infinite set, $\varphi\colon\Gamma\rightarrow\Gamma$ a function, and $w=(w_\gamma)_{\gamma\in\Gamma}\in\mathbb{Z}_\mathbbm{k}^\Gamma$ with $w_\gamma\not\equiv0$. Then we shall consider the ``weighted shift'' induced by $\varphi$ and $w$:
$$
\sigma\colon\mathbb{Z}_\mathbbm{k}^\Gamma\rightarrow\mathbb{Z}_\mathbbm{k}^\Gamma,\quad x=(x_\gamma)_{\gamma\in\Gamma}\mapsto\sigma x=(w_\gamma x_{\varphi\gamma})_{\gamma\in\Gamma},\textrm{ i.e.}, \sigma x_\gamma=w_\gamma x_{\varphi\gamma}\ \forall \gamma\in\Gamma.
$$
See \cite{Sh}. Here the phase space $\mathbb{Z}_\mathbbm{k}^\Gamma$ is equipped with the product topology. The classical case is that $\Gamma=\mathbb{Z}_+$ or $\mathbb{Z}$, $w_\gamma\equiv1$, and $\varphi\colon \gamma\mapsto\gamma+1$ (see Example~\ref{6.1}).
Note that
$$
\sigma^nx_\gamma=w_\gamma\dotsm w_{\varphi^{n-1}\gamma}x_{\varphi^n\gamma}\quad \forall x\in\mathbb{Z}_\mathbbm{k}^\Gamma, \gamma\in\Gamma\textrm{ and }n\ge1.
$$
An element $k\in\mathbb{Z}_\mathbbm{k}$ is said to be ``invertible'' iff there exists an element in $\mathbb{Z}_\mathbbm{k}$, denoted $k^{-1}$, such that $kk^{-1}=1\pmod{\mathbbm{k}}$. For example, $2^{-1}=3$ and $3^{-1}=2$ in $\mathbb{Z}_5$. We say $w\in\mathbb{Z}_\mathbbm{k}^\Gamma$ is ``invertible'' if $w_\gamma$ is invertible for all $\gamma\in\Gamma$.
A point $\gamma\in\Gamma$ is called ``periodic'' under $(\varphi,\Gamma)$, denoted $\gamma\in\texttt{Per}(\varphi)$, iff $\varphi^\tau\gamma=\gamma$ for some integer $\tau\ge1$. By $\texttt{Per}\,(\sigma)$ we denote the set of periodic points under $(\sigma,\mathbb{Z}_\mathbbm{k}^\Gamma)$.

\begin{6.3a}[{cf.~\cite[Thm.~4.1]{Sh} by a complicated proof}]
Let $\varphi\colon\Gamma\rightarrow\Gamma$ is 1-1 and let $w$ be invertible. Then:
\begin{enumerate}[$(1)$]
\item $\sigma\colon\mathbb{Z}_\mathbbm{k}^\Gamma\rightarrow\mathbb{Z}_\mathbbm{k}^\Gamma$ is continuous onto.
\item $\texttt{Per}\,(\sigma)$ is dense in $\mathbb{Z}_\mathbbm{k}^\Gamma$.
\end{enumerate}
\end{6.3a}

\begin{proof}
\item (1): The continuity is obvious. Now let $y=(y_\gamma)_{\gamma\in\Gamma}\in\mathbb{Z}_\mathbbm{k}^\Gamma$. Put $x_\gamma=0$ if $\gamma\notin\varphi(\Gamma)$, and, $x_\gamma=w_\alpha^{-1}y_\alpha$ if $\alpha\in\Gamma$ and $\gamma=\varphi\alpha$. Then $x=(x_\gamma)_{\gamma\in\Gamma}\in\mathbb{Z}_\mathbbm{k}^\Gamma$ such that $\sigma x=y$.

\item (2): Let $z=(z_\gamma)_{\gamma\in\Gamma}\in\mathbb{Z}_\mathbbm{k}^\Gamma$ and let $\gamma_1,\dotsc,\gamma_n\in\Gamma$ be pairwise distinct. It suffices to show there is a $\sigma$-periodic point in the clopen cylinder neighborhood of $z$:
$$U:=[z_{\gamma_1}]_{\gamma_1}\cap\dotsm\cap[z_{\gamma_n}]_{\gamma_n}=\{x=(x_\gamma)_{\gamma\in\Gamma}\,|\,x_{\gamma_1}=z_{\gamma_1}, \dotsc, x_{\gamma_n}=z_{\gamma_n}\}.$$
Since $\varphi$ is 1-1, there is no loss of generality in considering only the special case that $n=4$ and that $\gamma_1$ is a periodic point of $\varphi$ and $\gamma_2, \gamma_3, \gamma_4$ are non-periodic points with some integer $m\ge1$ such that
$\{\gamma_3,\ \varphi\gamma_3,\ \dotsc,\ \varphi^m\gamma_3=\gamma_2,\ \varphi^{m+1}\gamma_3=\varphi\gamma_2,\  \varphi^{m+2}\gamma_3=\varphi^2\gamma_2,\ \varphi^{m+3}\gamma_3=\varphi^3\gamma_2,\ \dotsc\}
$ and $\{\gamma_4,\ \varphi\gamma_4,\ \varphi^2\gamma_4,\ \dotsc\}$
have no common elements, for the other cases may be proved analogously.

Notice that $r\in\mathbb{Z}_\mathbbm{k}$ is invertible iff there exists an integer $\ell\ge1$ such that $r^\ell=1\pmod{\mathbbm{k}}$. \footnote{Sufficiency is obvious. Now conversely, assume $r$ is invertible. Clearly, $r^\ell$, for $\ell\ge1$, is invertible. Moreover, there exists a sequence $\ell_1, \ell_2, \dotsc\to\infty$ such that $r^{\ell_1}=r^{\ell_2}=\dotsm\pmod{\mathbbm{k}}$. Thus, $r^\ell=1\pmod{\mathbbm{k}}$ for some $\ell\ge1$.}
Then we can find an integer $\tau$ such that
$$
\tau\ge m+1,\quad \varphi^\tau\gamma_1=\gamma_1,\quad w_{\varphi^{\tau-1}\gamma_1}\dotsm w_{\gamma_1}=1.
$$
Let $\Gamma_0=\{\gamma_1,\dotsc,\varphi^{\tau-1}\gamma_1\}\cup\{\varphi^{i\tau}\gamma_j\,|\,j=2,3,4; i\ge0\}$. We now define a point $x=(x_\gamma)_{\gamma\in\Gamma}\in\mathbb{Z}_\mathbbm{k}^\Gamma$ as follows:
\begin{enumerate}[(a)]
\item $x_{\gamma}=z_{\gamma}$ for $\gamma\in\{\gamma_1, \varphi\gamma_1, \dotsc, \varphi^{\tau-1}\gamma_1\}$;
\item $x_{\gamma_2}=z_{\gamma_2}$ and $x_{\varphi^{i\tau}\gamma_2}=(w_{\varphi^{i\tau-1}\gamma_2}\dotsm w_{\gamma_2})^{-1}z_{\gamma_2}$ for $i=1,2,\dotsc$;
\item $x_{\gamma_3}=z_{\gamma_3}$ and $x_{\varphi^{i\tau}\gamma_3}=(w_{\varphi^{i\tau-1}\gamma_3}\dotsm w_{\gamma_3})^{-1}z_{\gamma_3}$ for $i=1,2,\dotsc$;
\item $x_{\gamma_4}=z_{\gamma_4}$ and $x_{\varphi^{i\tau}\gamma_4}=(w_{\varphi^{i\tau-1}\gamma_4}\dotsm w_{\gamma_4})^{-1}z_{\gamma_4}$ for $i=1,2,\dotsc$;
\item if $\gamma\notin\Gamma_0$ such that $\varphi^\tau\gamma=\gamma_j$ for some $j=2,3,4$, then $x_\gamma=(w_{\varphi^{\tau-1}\gamma}\dotsm w_\gamma)^{-1}z_{\gamma_j}$;
\item $x_\gamma=0$ for the other $\gamma$ in $\Gamma$.
\end{enumerate}
It is clear that $\sigma^\tau x=x$ and $x\in U$. The proof is complete.
\end{proof}

\begin{6.3b}[{cf.~\cite[Thm.~5.2]{Sh} for proving $(\sigma,\mathbb{Z}_\mathbbm{k}^\Gamma)$ to be T.T.}]
Let $\varphi\colon\Gamma\rightarrow\Gamma$ be 1-1 without periodic points; let $w$ be invertible. Then $(\sigma,\mathbb{Z}_\mathbbm{k}^\Gamma)$ is weakly mixing as a $\mathbb{Z}_+$-semiflow; and moreover, $(\sigma,\mathbb{Z}_\mathbbm{k}^\Gamma\times\mathbb{Z}_\mathbbm{k}^\Gamma)$ is Devaney chaotic (T.T. with a dense set of periodic points).
\end{6.3b}

\begin{proof}
Let
$$
U_1=[k_1]_{\gamma_1}\cap\dotsm\cap[k_n]_{\gamma_n},\quad U_2=[i_1]_{\alpha_1}\cap\dotsm\cap[i_n]_{\alpha_n},
$$
and
$$
V_1=[j_1]_{\delta_1}\cap\dotsm\cap[j_m]_{\delta_m},\quad V_2=[\ell_1]_{\lambda_1}\cap\dotsm\cap[\ell_m]_{\lambda_m}
$$
be cylinder subsets of $\mathbb{Z}_\mathbbm{k}^\Gamma$. To show $(\sigma,\mathbb{Z}_\mathbbm{k}^\Gamma)$ is weakly mixing, it suffices to prove that there exists an integer $\tau>0$ such that
$\sigma^\tau(U_1\times U_2)\cap(V_1\times V_2)\not=\emptyset$.
We notice that since $\varphi$ is 1-1 and has no periodic points, there exists an integer $\tau\ge1$ such that
$$
\{\gamma_1,\dotsc,\gamma_n,\alpha_1,\dotsc,\alpha_n\}\cap\left(\{\varphi^i\delta_j\,|\,i\ge\tau, 1\le j\le m\}\cup\{\varphi^i\lambda_j\,|\,i\ge\tau, 1\le j\le m\}\right)=\emptyset.
$$
Let $x=(x_\gamma)_{\gamma\in\Gamma}\in U_1$ and $y=(y_\gamma)_{\gamma\in\Gamma}\in U_2$ be defined as follows:
$$
x_\gamma=
\begin{cases}
k_i& \textrm{if }\gamma=\gamma_i\textrm{ for }1\le i\le n,\\
\left(w_{\varphi^{\tau-1}\delta_i}\dotsm w_{\delta_i}\right)^{-1}j_i& \textrm{if }\gamma=\varphi^\tau\delta_i\textrm{ for }1\le i\le m,\\
0 & \textrm{for other }\gamma\in\Gamma.
\end{cases}
$$
and
$$
y_\gamma=
\begin{cases}
i_j& \textrm{if }\gamma=\alpha_j\textrm{ for }1\le j\le n,\\
\left(w_{\varphi^{\tau-1}\lambda_j}\dotsm w_{\lambda_j}\right)^{-1}\ell_j& \textrm{if }\gamma=\varphi^\tau\lambda_j\textrm{ for }1\le j\le m,\\
0 & \textrm{for other }\gamma\in\Gamma.
\end{cases}
$$
Clearly, $\sigma^\tau x\in V_1$ and $\sigma^\tau y\in V_2$. Thus $(x,y)\in U_1\times U_2$ such that $\sigma^\tau(x,y)\in V_1\times V_2$.

This proves that $(\sigma,\mathbb{Z}_\mathbbm{k}^\Gamma)$ is weakly mixing. By Lemma~\ref{6.3}a, $(\sigma,\mathbb{Z}_\mathbbm{k}^\Gamma\times\mathbb{Z}_\mathbbm{k}^\Gamma)$ has a dense set of periodic points. Thus $(\sigma,\mathbb{Z}_\mathbbm{k}^\Gamma\times\mathbb{Z}_\mathbbm{k}^\Gamma)$ is Devaney chaotic.
\end{proof}

In fact, it is easy to verify that if $(\sigma,\mathbb{Z}_\mathbbm{k}^\Gamma)$ is weakly mixing, then $\sigma$ is onto so that $\varphi$ is 1-1 without periodic points and $w$ is invertible.

\begin{6.3c}
Let $\Gamma$ be countable. Let $\varphi\colon\Gamma\rightarrow\Gamma$ be 1-1 without periodic points; let $w$ be invertible. Then there exists a dense uncountable Li-Yorke scrambled subset of $\mathbb{Z}_\mathbbm{k}^\Gamma$ rel. $\mathbb{Z}_\mathbbm{k}^\Gamma\times\mathbb{Z}_\mathbbm{k}^\Gamma$.
\end{6.3c}

\begin{proof}
This follows easily from Theorem~\ref{6.3}b and Proposition~\ref{5.6}.
\end{proof}

Recall that $(x,y)\in\mathbb{Z}_\mathbbm{k}^\Gamma\times\mathbb{Z}_\mathbbm{k}^\Gamma$ is called an ``asymptotic pair'' if for all $\varepsilon\in\mathscr{U}_{\mathbb{Z}_\mathbbm{k}^\Gamma}$ there exists an integer $\tau>0$ such that $(\sigma^nx,\sigma^ny)\in\varepsilon$ for all $n\ge\tau$.

\begin{6.3d}
Let $\varphi\colon\Gamma\rightarrow\Gamma$ be 1-1 having a non-periodic point and let $w$ be invertible. Then there exists a  uncountable subset $\Theta$ of $\mathbb{Z}_\mathbbm{k}^\Gamma$ such that $(x,y)$ is a proximal non-asymptotic pair under $(\sigma,\mathbb{Z}_\mathbbm{k}^\Gamma)$ for all $x,y\in\Theta$ with $x\not=y$.
\end{6.3d}

\begin{proof}
Let $\theta\in\Gamma$ be a non-periodic point of $\varphi$. Since $\varphi$ is 1-1, then $\Gamma_\theta:=\{\gamma,\varphi\gamma,\varphi^2\gamma,\dotsc\}$ is a countable infinite $\varphi$-invariant pairwise distinct subset of $\Gamma$.
By considering the naturally defined system $(\sigma,\mathbb{Z}_\mathbbm{k}^{\Gamma_\theta})$ based on $\Gamma_\theta\xrightarrow{\gamma\mapsto\varphi\gamma}\Gamma_\theta$ and $\Gamma_\theta\xrightarrow{\gamma\mapsto w_\gamma}\mathbb{Z}_\mathbbm{k}$, and by Proposition~\ref{6.3}c, it follows that there exists a uncountable Li-Yorke scrambled set $\Xi$ in $\mathbb{Z}_\mathbbm{k}^{\Gamma_\theta}$ rel. $\mathbb{Z}_\mathbbm{k}^{\Gamma_\theta}\times\mathbb{Z}_\mathbbm{k}^{\Gamma_\theta}$.
For all $x=(x_\gamma)_{\gamma\in\Gamma_\theta}\in\Xi$, define $x^\prime\in\mathbb{Z}_\mathbbm{k}^\Gamma$ by $x_\gamma^\prime=x_\gamma$ if $\gamma\in\Gamma_\theta$ and $x_\gamma^\prime=0$ if $\gamma\in\Gamma\setminus\Gamma_\theta$. Let $\Theta=\{x^\prime\,|\,x\in\Xi\}$. Clearly $\Theta$ possesses the desired property. The proof is complete.
\end{proof}

\begin{6.3e}
The condition $w$ is invertible is critical for Lemma~\ref{6.3}a, Theorem~\ref{6.3}b, Propositions~\ref{6.3}c and \ref{6.3}d. Otherwise we can construct a counterexample as follows: Let $\mathbbm{k}=4$, $w_\gamma\equiv 2$ $(\not=0)$ for all $\gamma\in\Gamma$, and $o=(o_\gamma)_{\gamma\in\Gamma}\in\mathbb{Z}_4^\Gamma$ with $o_\gamma=0$ for $\gamma\in\Gamma$. Now let $x\in\mathbb{Z}_4^\Gamma$. Since $2\mathbb{Z}_4=\{0,2\}$, then $\sigma x\in\{0,2\}^\Gamma$, $\sigma^2x=o$, and $\sigma^nx=0$ for all $n\ge3$,
for all 1-1 function $\varphi\colon\Gamma\rightarrow\Gamma$. Thus, $(\sigma,\mathbb{Z}_\mathbbm{k}^\Gamma)$ has no any chaotic dynamics.
\end{6.3e}
\end{exa}
\bigskip
\noindent
\textbf{Acknowledgments. }
This work was supported by National Natural Science Foundation of China (Grant No. 11790274) and PAPD of Jiangsu Higher Education Institutions.

\end{document}